\documentclass[a4paper,oneside,english]{amsart}
\usepackage{libertineRoman}
\usepackage{avant}
\usepackage{libertineMono}
\usepackage[T1]{fontenc}
\usepackage[latin9]{inputenc}
\usepackage{babel}
\usepackage{pifont}
\usepackage{textcomp}
\usepackage{mathtools}
\usepackage{enumitem}
\usepackage{amstext}
\usepackage{amsthm}
\usepackage{amssymb}
\usepackage{setspace}
\usepackage{xargs}[2008/03/08]
\usepackage[unicode=true,pdfusetitle,
	bookmarks=true,bookmarksnumbered=false,bookmarksopen=false,
	breaklinks=false,pdfborder={0 0 0},pdfborderstyle={},backref=false,colorlinks=false]
{hyperref}

\makeatletter

\numberwithin{equation}{section}
\numberwithin{figure}{section}
\theoremstyle{plain}
\newtheorem*{thm*}{\protect\theoremname}
\theoremstyle{plain}
\newtheorem{thm}{\protect\theoremname}
\theoremstyle{definition}
\newtheorem{defn}[thm]{\protect\definitionname}
\theoremstyle{plain}
\newtheorem{lem}[thm]{\protect\lemmaname}
\theoremstyle{remark}
\newtheorem{rem}[thm]{\protect\remarkname}
\theoremstyle{remark}
\newtheorem{notation}[thm]{\protect\notationname}
\theoremstyle{remark}
\newtheorem{claim}[thm]{\protect\claimname}
\theoremstyle{plain}
\newtheorem{cor}[thm]{\protect\corollaryname}
\newlist{casenv}{enumerate}{4}
\setlist[casenv]{leftmargin=*,align=left,widest={iiii}}
\setlist[casenv,1]{label={{\itshape\ \casename} \arabic*.},ref=\arabic*}
\setlist[casenv,2]{label={{\itshape\ \casename} \roman*.},ref=\roman*}
\setlist[casenv,3]{label={{\itshape\ \casename\ \alph*.}},ref=\alph*}
\setlist[casenv,4]{label={{\itshape\ \casename} \arabic*.},ref=\arabic*}

\AtBeginDocument{

}

\makeatother

\providecommand{\casename}{Case}
\providecommand{\claimname}{Claim}
\providecommand{\corollaryname}{Corollary}
\providecommand{\definitionname}{Definition}
\providecommand{\lemmaname}{Lemma}
\providecommand{\notationname}{Notation}
\providecommand{\remarkname}{Remark}
\providecommand{\theoremname}{Theorem}

\newcommand\power{\mathcal{P}}%
\newcommand\LL{\mathcal{L}}%
\newcommand\zfc{\mathrm{ZFC}}%

\newcommand\dom{\operatorname{Dom}}%
\newcommand\rng{\operatorname{Range}}%
\newcommand\ord{{\bf Ord}}%
\newcommand\lp{\operatorname{lp}}%
\newcommand\meas{\operatorname{meas}}%
\newcommandx\sing[1][usedefault, addprefix=\global, 1=]{{\bf Sing}_{#1}}%
\newcommand\lex{\mathrm{Lex}}%
\newcommand\id{\operatorname{Id}}%
\newcommand\ult{\operatorname{Ult}}%
\newcommand\It{\operatorname{It}}%
\newcommand\cof[1]{\mathrm{cf}\left(#1\right)}%
\newcommand\otp{\operatorname{otp}}%
\newcommand\psu[2]{\prescript{#1}{}{#2}}%
\newcommand\con{\subseteq}%
\newcommand\til{,\dots,}%
\newcommand\smin{\mathord{\smallsetminus}}%
\newcommand\mets{\mathord{\upharpoonright}}%

\title[Models for short sequences of measures in $C^{*}$]{Models for short sequences of measures \\ in the cofinality-$\omega$
	constructible model}
\author{Ur Ya'ar}
\address{\noindent Einstein Institute of Mathematics \newline Hebrew University
	of Jerusalem\newline Edmond J. Safra Campus, Givat Ram \newline
	Jerusalem 91904, ISRAEL}
\email{ur.yaar@mail.huji.ac.il}
\thanks{I would like to thank my advisor, Prof. Menachem Magidor, for his
guidance and support without which this work would not have been possible.
I would also like to thank Jouko V\"{a}\"{a}n\"{a}nen and Otto Rajala
for reading and commenting on a previous draft of this paper.}

\makeatletter
\@namedef{subjclassname@2020}{%
	\textup{2020} Mathematics Subject Classification}
\makeatother
\subjclass{03E45 (Primary) 03E47, 03E55, 03E70 (Secondary)}
\keywords{Inner models, Cofinality logic, Measurable cardinal, Generalized constructibility, Core model}

\begin{document}
\begin{abstract}
	We investigate the relation between $C^{*}$, the model of sets constructible
	using first order logic augmented with the ``cofinality-$\omega$''
	quantifier, and ``short'' sequences of measures -- sequences of
	measures of order $1$, which are shorter than their minimum. We show
	that certain core models for short sequences of measures are contained
	in $C^{*}$; we compute $C^{*}$ in a model of the form $L\left[\mathcal{U}\right]$
	where $\mathcal{U}$ is a short sequence of measures, and in models
	of the form $L\left[\mathcal{U}\right]\left[G\right]$ where $G$
	is generic for adding Prikry sequences to some of the measurables
	of $\mathcal{U}$; and prove that if there is an inner model with
	a short sequence of measures of order type $\chi$, then there is
	such an inner model in $C^{*}$.
\end{abstract}
\maketitle
\section{Introduction}

The model $C^{*}$, introduced by Kennedy, Magidor and V\"{a}\"{a}n\"{a}nen
in \cite{IMEL}, is the model of sets constructible using the logic
$\LL(Q_{\omega}^{\mathrm{cf}})$ -- first order logic augmented with
the ``cofinality-$\omega$'' quantifier. As shown in \cite{IMEL},
this is the same as $L[\ord_{\omega}]$ where $\ord_{\omega}=\{\alpha\in\ord\mid\cof{\alpha}=\omega\}$.
Like $L$ -- the model of sets constructible using just first order
logic -- this is a model of $\zfc$. It is however in many cases
much larger than $L$, for example in the presence of large cardinals.
Specifically, Kennedy, Magidor and V\"{a}\"{a}n\"{a}nen  prove in
\cite{IMEL} the following:
\begin{thm*}
	\begin{enumerate}
		\item $C^{*}$ contains the Dodd-Jensen core model.
		\item If there is an inner model for a measurable cardinal, then $C^{*}$
		      contains an inner model for a measurable cardinal.
		\item If $V=L^{\mu}$ for some measure $\mu$ then $C^{*}=M_{\omega^{2}}\left[\left\langle \kappa_{\omega\cdot n}\mid n<\omega\right\rangle \right]$
		      where $M_{\omega^{2}}$ is the $\omega^{2}$-iterate of $V$ and $\kappa_{\omega\cdot n}$
		      is the $\omega\cdot n$ image of the measurable under the iteration
		      embedding.
	\end{enumerate}
\end{thm*}
Our aim in this paper is to generalize this theorem to obtain inner
models with more measurable cardinals. Our focus here is on ``short''
sequences of measures -- a notion introduced by Koepke in \cite{koepkedoc,koepke1988}
-- sequences of measures of order $1$, which are shorter than their
minimum (so the sequence doesn't ``stretch'' when taking ultrapowers).
In section \ref{sec:Short-mice} we define the notions of \emph{short
	mice} and \emph{short core models}, and present some basic facts concerning
them; in section \ref{sec:Short-mice-C-star} we show that certain
short core models are contained in $C^{*}$; in section \ref{sec:The-C-star-of}
we compute $C^{*}$ in a model of the form $L\left[\mathcal{U}\right]$
where $\mathcal{U}$ is a short sequence of measures, and in models
of the form $L\left[\mathcal{U}\right]\left[G\right]$ where $G$
is generic for adding Prikry sequences to some of the measurables
of $\mathcal{U}$; and in section \ref{sec:Inner-models-in-C-star}
we prove that if there is an inner model with a short sequence of
measures of order type $\chi$, then there is such an inner model
in $C^{*}$.

It is interesting to note that while we can get inner models with
many measurable cardinals contained in $C^{*}$, it is yet unknown
whether $C^{*}$ can actually contain even a single measurable cardinal.
By \cite[theorem 5.7]{IMEL}, if $V=C^{*}$ then there is no measurable
cardinal in $V$, however $C^{*}$ does not necessarily satisfy $V=C^{*}$,
so it might still possible to have a measurable cardinal in $C^{*}$,
which would simply imply that $\left(C^{*}\right)^{C^{*}}\ne C^{*}$.

\section{\label{sec:Short-mice}Short mice }

First we cite some definitions from \cite{koepke1988} (as they appear
in \cite{lucke-schlicht}).
\begin{defn}
	Let $D$ be a class. $D$ is simple if the following statements hold:
	\begin{enumerate}
		\item If $x\in D$, then $x=\langle\delta,a\rangle$ for some $\delta\in\ord$
		      and $a\subseteq\delta$;
		\item If $\langle\delta,a\rangle\in D$, then $\langle\delta,\delta\rangle\in D$.
	\end{enumerate}
	For a simple $D$, we define $\dom(D)=\{\delta\mid\langle\delta,\delta\rangle\in D\}$
	and $D(\delta)=\{a\subseteq\delta\mid\langle\delta,a\rangle\in D\}$
	for all $\delta\in\dom(D)$. $D$ is called \emph{a sequence of measures}
	if $D$ is simple and $D(\delta)$ is a normal ultrafilter on $\delta$
	for every $\delta\in\dom(D)$
\end{defn}

\begin{defn}
	Let $D$ be a simple set.
	\begin{enumerate}
		\item We say that $M=\left\langle |M|,F_{M}\right\rangle $ is a \emph{premouse
			      over $D$} if the following statements hold:
		      \begin{enumerate}
			      \item $|M|$ is a transitive set and $F_{M}$ is a simple set with
			            \[
				            \sup(\dom(D))<\min\left(\dom\left(F_{M}\right)\right)\in|M|;
			            \]
			      \item $\left\langle |M|,\in,F_{M}\right\rangle \vDash$``$F_{M}$ is a sequence
			            of measures'';
			      \item $|M|=J_{\alpha_{M}}\left[D,F_{M}\right]$ for some ordinal $\alpha_{M}$
			            with $\omega\alpha_{M}=|M|\cap\ord$.
		      \end{enumerate}
		\item If $M$ is a premouse over $D$, we denote \emph{the measurables of
			      $M$} by
		      \[
			      \meas(M):=\dom\left(F_{M}\right)\cap(\omega\alpha_{M})
		      \]
		      and \emph{the lower part of $M$} by
		      \[
			      \lp(M):=H(\min(\meas(M)))^{|M|}
		      \]
		      (where $H(\theta)$ is the set of sets hereditarily of size $<\theta$).
		\item Given a premouse $M$ over $D$ and $\delta\in\meas(M)$, a premouse
		      $N$ over $D$ is the \emph{ultrapower} of $M$ at $\delta$ if there
		      is a unique map $j:|M|\to|N|$ with the following properties:
		      \begin{enumerate}
			      \item $j:\left\langle |M|,\in,D,F_{M}\right\rangle \to\left\langle |N|,\in,D,F_{N}\right\rangle $
			            is $\Sigma_{1}$-elementary.
			      \item $|N|=\left\{ j(f)(\delta)\mid f\in\psu{\delta}{|M|}\cap|M|\right\} $.
			      \item $F_{M}(\delta)\cap|M|=\{x\in\power(\delta)\cap|M|\mid\delta\in j(x)\}$.
		      \end{enumerate}
		      If such a premouse exists, we denote it by $\ult(M,F(\delta))$ and
		      we call the corresponding map $j$ the ultrapower embedding of $M$
		      at $\delta.$
		\item Given a premouse $M$ over $D$ and a function $I:\lambda\to\ord$
		      with $\lambda\in\ord$, a system
		      \[
			      \It(M,I)=\left\langle \left\langle M_{\alpha}\mid\alpha\leq\lambda\right\rangle ,\left\langle j_{\alpha,\beta}\mid\alpha\leq\beta\leq\lambda\right\rangle \right\rangle
		      \]
		      is called the\emph{ iterated ultrapower of $M$} by $I$ if the following
		      statements hold for all $\gamma\leq\lambda:$
		      \begin{enumerate}
			      \item $M=M_{0}$ and $M_{\gamma}$ is a premouse over $D$.
			      \item Given $\alpha\leq\beta\leq\gamma,\:j_{\beta,\gamma}:\left|M_{\beta}\right|\to\left|M_{\gamma}\right|$
			            is a function, $j_{\gamma,\gamma}=\mathrm{id}_{\left|M_{\gamma}\right|}$
			            and $j_{\alpha,\gamma}=j_{\beta,\gamma}\circ j_{\alpha,\beta}.$
			      \item If $\gamma<\lambda$ and $I(\gamma)\in\meas\left(M_{\gamma}\right)$,
			            then
			            \[
				            M_{\gamma+1}=\ult\left(M_{\gamma},F_{M_{\gamma}}(I(\gamma))\right)
			            \]
			            and $j_{\gamma,\gamma+1}$ is the ultrapower embedding of $M_{\gamma}$
			            at $F_{M_{\gamma}}(I(\gamma))$. In the other case, if $\gamma<\lambda$
			            and $I(\gamma)\notin$ meas $\left(M_{\gamma}\right)$, then $M_{\gamma}=M_{\gamma+1}$
			            and $j_{\gamma,\gamma+1}=\mathrm{id}_{\left|M_{\gamma}\right|}$
			      \item If $\gamma$ is a limit ordinal, then
			            \[
				            \left\langle \left\langle M_{\gamma},\in,D,F_{M_{\gamma}}\right\rangle ,\left\langle j_{\beta,\gamma}\mid\beta<\gamma\right\rangle \right\rangle
			            \]
			            is a direct limit of the directed system
			            \[
				            \left\langle \left\langle M_{\beta},\in,D,F_{M_{\beta}}\mid\beta<\gamma\right\rangle ,\left\langle j_{\alpha,\beta}\mid\alpha\leq\beta<\gamma\right\rangle \right\rangle .
			            \]

		      \end{enumerate}
		      In this situation, we let $M_{I}$ denote $M_{\lambda}$ and let $j_{I}$
		      denote $j_{0,\lambda}$.
		\item A premouse $M$ over $D$ is\emph{ iterable} if the system $\It(M,I)$
		      is well-founded for every function $I:\lambda\to\ord$ with $\lambda\in\ord$.
		\item A premouse $M$ over $D$ is \emph{short} if one of the following
		      statements holds:
		      \begin{enumerate}
			      \item $D=\varnothing$ and $\otp(\meas(M)\cap\gamma)<\min(\meas(M))$ for
			            all $\gamma\in|M|\cap\ord$.
			      \item $D\neq\varnothing$ and $\otp(\meas(M))<\min(\dom(D))$.
		      \end{enumerate}
	\end{enumerate}
	At this stage we diverge from Koepke's terminology and pick up Mitchell's
	notion of a mouse from \cite[§3]{mitchell1984core}, in order to be
	able to use a stronger comparison lemma. Note that for short sequences,
	the notion of coherence (definition 1.6 in \cite{mitchell1984core})
	is trivial.
	\begin{enumerate}[resume]
		\item \cite[definition 3.1]{mitchell1984core} A premouse $M=J_{\alpha}\left[D,F\right]$
		      is called a \emph{(short) $D$-mouse }if it is short, iterable and
		      satisfies that there is some $\rho$, $\sup(D)\leq\rho<\min(F)$ and
		      a parameter $p\in\left(\alpha\smin\rho\right)^{<\omega}$ such that
		      $M=\mathcal{H}_{1}^{M}(\rho\cup p)$ (the $\Sigma_{1}$ Skolem hull
		      in $M$).

		      Denote by $\rho_{M}$ the least $\rho$ for which such a $p$ exists
		      and $p_{M}$ the least such parameter (under the usual parameter well-order).
		\item \cite[definition 3.2]{mitchell1984core} If $M,N$ are (short) $D$-mice,
		      denote $M<_{D}N$ if there are iterated ultrapowers $i:M\to J_{\alpha}\left[D,F\right]$,
		      $j:N\to J_{\beta}\left[D,G\right]$ such that $i\mets\rho_{M}$ and
		      $j\mets\rho_{N}$ are the identity, $\alpha\leq\beta$, $G\mets\alpha=F$
		      and $\left(\alpha,\rho_{M},i(p_{M})\right)<\left(\beta,\rho_{N},j(p_{N})\right)$
		      lexicographically.
	\end{enumerate}
\end{defn}

\begin{lem}[{\cite[theorem 3.3]{mitchell1984core}}]
	\label{lem:mice order} $<_{D}$ is a well ordering of (short) $D$
	mice such that $M<_{D}N$ implies $M\in L(N)$.
\end{lem}

\begin{rem}
	The original theorem in \cite{mitchell1984core} discusses general
	mice, and the proof uses comparison of mice by iterations. Since iterations
	of short mice preserves shortness it holds when restricting to short
	mice as well.
\end{rem}

\begin{lem}
	\label{lem:fixed cof}If $M$ is a $D$-mouse then for every $\eta\in\left(\rho_{M},\omega\alpha_{M}\right)$
	such that $M\vDash\eta$ is regular, $\mathrm{cf}^{V}(\eta)=\mathrm{cf}^{V}(\omega\alpha_{M})$.
\end{lem}

\begin{proof}
	$M=J_{\alpha_{M}}\left[D,F\right]$ for some $F$. Let $\eta\in\left(\rho_{M},\omega\alpha_{M}\right)$
	such that $M\vDash\eta$ is regular.

	Note that there is an ascending sequence of transitive sets
	\[
		\left\langle A_{\beta}\mid\beta<\mathrm{cf}^{V}(\omega\alpha_{M})\right\rangle
	\]
	such that $A_{\beta}\in M$ for every $\beta$ and $M=\bigcup_{\beta<\mathrm{cf}^{V}(\omega\alpha_{M})}A_{\beta}$:
	if $\alpha_{M}$ is limit, then $J_{\alpha_{M}}\left[D,F\right]=\bigcup_{\beta<\alpha_{M}}J_{\beta}\left[D,F\right]$
	so we can take some cofinal sequence of $\beta$'s in $\alpha_{M}$
	($\mathrm{cf}^{V}(\alpha_{M})=\mathrm{cf}^{V}(\omega\alpha_{M})$
	in this case). If $\alpha_{M}=\alpha+1$, then $J_{\alpha_{M}}\left[D,F\right]=\bigcup_{n<\omega}S_{\omega\alpha+n}$
	where $S_{\gamma}$ is the auxiliary Jensen hierarchy (cf. \cite[pg. 610]{handbook-finestructure}),
	where each successor step adds the images of the previous step under
	finitely many rudimentary functions, and in this case $\mathrm{cf}^{V}(\omega\alpha_{M})=\mathrm{cf}^{V}(\omega\alpha+\omega)=\omega$.

	For every $\beta<\mathrm{cf}^{V}(\omega\alpha_{M})$, $M$ has a Skolem
	function $h_{\beta}$ for $\left\langle A_{\beta},\in,D,F\right\rangle $,
	and since $M\vDash\eta$ regular, and $\rho_{M}<\eta$, $h_{\beta}``\left[\rho_{M}\cup p_{M}\right]^{<\omega}$
	is in $M$ of size $<\eta$, so
	\[
		\alpha_{\beta}:=\sup\left(h_{\beta}``\left[\rho_{M}\cup p_{M}\right]^{<\omega}\cap\eta\right)<\eta
	\]
	But since $M=\mathcal{H}_{1}^{M}(\rho_{M}\cup p_{M})$, each element
	of $\eta$ is in $h_{\beta}``\left(\rho_{M}\cup p_{M}\right)$ for
	some $\beta<\mathrm{cf}^{V}(\omega\alpha_{M})$ (if an element satisfies
	some $\Sigma_{1}$ formula in $M$, then it satisfies this formula
	in some initial segment of $M$), so the sequence $\left\langle \alpha_{\beta}\mid\beta<\mathrm{cf}^{V}(\omega\alpha_{M})\right\rangle $
	is cofinal in $\eta$, hence $\mathrm{cf}^{V}(\eta)=\mathrm{cf}^{V}(\omega\alpha_{M})$.
\end{proof}
\begin{notation}
	For a $D$-mouse $M$, denote by $\kappa(\gamma)$ be the $\gamma$th
	member of $\dom(F_{M})$ and $\xi(\gamma)=\left(\kappa(\gamma)^{+}\right)^{M}$.
	Note that for every $\gamma$ we have $\xi(\gamma)>\rho_{M}$ and
	$M\vDash\xi(\gamma)$ is regular so by lemma \ref{lem:fixed cof}
	$\mathrm{cf}^{V}\left(\xi(\gamma)\right)=\mathrm{cf}^{V}\left(\alpha_{M}\right)$.
	Denote this by $\delta_{M}$.
\end{notation}

We now establish some notation and terminology for iterations of
short $D$-mice.
\begin{notation}
	Let $M=\left\langle \left|M\right|,D,F\right\rangle $ be a short
	$D$-mouse, $\chi\leq\otp\dom(F)$, and $\vec{\theta}=\left\langle \theta_{\gamma}\mid\gamma<\chi\right\rangle $
	some sequence of ordinals. We define the \emph{iteration of $M$ according
		to $\vec{\theta}$} as follows. Set $M_{0}^{0}=M$, $F_{0}^{0}=F$
	and define an iteration $\left\langle M_{\beta}^{\gamma},i_{\alpha,\beta}^{\gamma},j^{\gamma}\mid\gamma<\chi,\text{\ensuremath{\alpha\leq\beta<\theta_{\gamma}}}\right\rangle $
	inductively on $\gamma<\chi$ such that:
	\begin{itemize}
		\item $i_{\alpha,\beta}^{\gamma}:M_{\alpha}^{\gamma}\to M_{\beta}^{\gamma}$
		      and $j^{\gamma}:M_{0}^{0}\to M_{0}^{\gamma}$ are $\Sigma_{1}$-elementary;
		\item $i_{\beta,\beta}^{\gamma}=\id_{M_{\beta}^{\gamma}}$ and for $\alpha<\beta'\leq\beta$,
		      $i_{\alpha,\beta}^{\gamma}=i_{\beta',\beta}^{\gamma}\circ i_{\alpha,\beta'}^{\gamma}$
		      ;
		\item If $M_{\beta}^{\gamma}=\langle|M_{\beta}^{\gamma}|,D,F_{\beta}^{\gamma}\rangle$
		      we set $A_{\beta}^{\gamma}=\dom(F_{\beta}^{\gamma})$;

		      Note that these are iterations of short mice so $\otp\dom(F_{\beta}^{\gamma})=\otp\dom(F)$.
		\item $\kappa_{0}^{\gamma}$ is the $\gamma$th element of $A_{0}^{\gamma}$,
		      $\xi_{0}^{\gamma}=\big(\big(\kappa_{0}^{\gamma}\big)^{+}\big)^{M_{0}^{\gamma}}$
		      or (if this doesn't exist in $M_{0}^{\gamma}$) $\xi_{0}^{\gamma}=M_{0}^{\gamma}\cap\ord$,
		      $U_{0}^{\gamma}=F_{0}^{\gamma}(\kappa_{0}^{\gamma})$ is the ultrafilter
		      of $M_{0}^{\gamma}$ on $\kappa_{0}^{\gamma}$;
		\item If $M_{0}^{\gamma}$, $\kappa_{0}^{\gamma}$, $\xi_{0}^{\gamma}$,
		      $U_{0}^{\gamma}$ are defined, iterate $M_{0}^{\gamma}$ according
		      to $U_{0}^{\gamma}$ up to $\theta_{\gamma}$ to get the iteration
		      $\left\langle M_{\beta}^{\gamma},i_{\alpha,\beta}^{\gamma}\mid\alpha,\beta\leq\theta_{\gamma}\right\rangle $
		      and set $\kappa_{\beta}^{\gamma}=i_{0,\beta}^{\gamma}(\kappa_{0}^{\gamma})$,
		      $\xi_{\beta}^{\gamma}=i_{0,\beta}^{\gamma}(\xi_{0}^{\gamma})$, $U_{\beta}^{\gamma}=i_{0,\beta}^{\gamma}(U_{0}^{\gamma})$;
		\item Set $M_{0}^{\gamma+1}=M_{\theta_{\gamma}}^{\gamma}$, $j^{\gamma+1}=i_{0,\theta_{\gamma}}^{\gamma}\circ j^{\gamma}$;
		\item For limit $\gamma$, set $M_{0}^{\gamma},j^{\gamma}$ as the directed
		      limit of $\left\langle M_{0}^{\beta},j^{\beta}\mid\beta<\gamma\right\rangle $;
		\item Let $M^{\chi}$ be the directed limit of $\left\langle M_{0}^{\beta},j^{\beta}\mid\beta<\chi\right\rangle $
		      if $\chi$ is limit or $M_{\theta_{\gamma}}^{\gamma}$ if $\chi=\gamma+1$,
		      with $j^{\chi}:M_{0}^{0}\to M^{\chi}$ the corresponding embedding.
	\end{itemize}
\end{notation}

\begin{lem}
	\label{lem:fixed-succ-cof}In the above context, for every $\gamma<\chi$
	and $\beta<\theta_{\gamma}$ we have $\mathrm{cf}^{V}(\xi_{\beta}^{\gamma})=\mathrm{cf}^{V}\left(\xi(\gamma)\right)=\delta_{M}$.
\end{lem}

\begin{proof}
	From the definitions, we can see that $\xi_{\beta}^{\gamma}$ is in
	$M_{\beta}^{\gamma}$ the image of $\xi(\gamma)$ under an iteration
	embedding $j:=i_{0,\beta}^{\gamma}\circ j^{\gamma}$. Thus every $\eta<\xi_{\beta}^{\gamma}$
	is of the form $j(f)\left(\kappa_{1}\til\kappa_{n}\right)$ for $f\in M_{0}^{0}$,
	$f:\kappa(\gamma)^{n}\to\xi(\gamma)$ and some iteration points $\kappa_{1}<\dots<\kappa_{n}<\kappa_{\beta}^{\gamma}$.
	$f$ must be bounded in $\xi(\gamma)$ so if $\left\langle \gamma_{\alpha}\mid\alpha<\mathrm{cf}^{V}\left(\xi(\gamma)\right)\right\rangle $
	is cofinal in $\xi(\gamma)$, there is some $\alpha$ such that $f(\zeta)<\gamma_{\alpha}$
	for every $\zeta$ so $j(f)\left(\kappa_{1}\til\kappa_{n}\right)<j(\gamma_{\alpha})$
	hence $\left\langle j(\gamma_{\alpha})\mid\alpha<\mathrm{cf}^{V}\left(\xi(\gamma)\right)\right\rangle $
	is cofinal in $\xi_{\beta}^{\gamma}$, and since $\mathrm{cf}^{V}\left(\xi(\gamma)\right)$
	is regular we have $\mathrm{cf}^{V}(\xi(\gamma))=\mathrm{cf}^{V}(\xi_{\beta}^{\gamma})$.
\end{proof}
With this terminology we can easily state a useful version of the
comparison lemma proved by Koepke:
\begin{lem}[{\cite[lemma 5.17]{koepkedoc}}]
	\label{lem:comparison} Let $M$ be an iterable pre-mouse over $D$,
	$S$ an increasing sequence of regular cardinals greater than $\left|M\right|$
	with $\otp S=\otp\meas(M)$. Then the result of iterating $M$ according
	to $S$ is of the form $\bar{M}=J_{\alpha_{\bar{M}}}\left[D,\mathrm{CU}\mets S\right]$
	where $\mathrm{CU}$ is the closed unbounded simple predicate --
	$\left\langle \kappa,z\right\rangle \in\mathrm{CU}$ iff $\kappa$
	is a limit ordinal and $z\con\kappa$ contains a closed unbounded
	subset of $\kappa$.
\end{lem}

We finish this section with the notion of the core model for short
sequences of measures.
\begin{defn}
	Define $K\left[D\right]$ as the class of sets constructible from
	short $D$-mice.
\end{defn}

\begin{thm}
	\label{thm:K(D)}
	$K\left[D\right]=\bigcup\left\{ \lp(M)\mid M\text{ is a short \ensuremath{D}-mouse}\right\} $.
\end{thm}

\begin{rem}
	This is the way $K\left[D\right]$ is defined by Koepke in \cite{koepkedoc,koepke1988},
	so our slightly different definition yields the same result.
\end{rem}

\begin{proof}
	The inclusion $\supseteq$ is clear, so we need to show that for every
	$x\in K\left[D\right]$ there is some short $D$-mouse containing
	$x$ in its lower part. $x$ is constructed by using a \emph{set}
	of mice, $X$. Let $S$ be an ascending sequence of regular cardinals
	greater than $\lambda^{+}$ where
	\[
		\lambda=\sup\left(\left\{ \left|M\right|\mid M\in X\right\} \cup\left\{ \left|\mathrm{tr}(x)\right|,\left|X\right|\right\} \right),
	\]
	such that $\otp S=\sup\left\{ \otp\meas(M)\mid M\in X\right\} $.
	We iterate each $M\in X$ according to $S$ (i.e. for $\gamma<\otp\meas(M)$,
	we iterate $M$ by its $\gamma$th measure $S(\gamma)$ many times),
	so by lemma \ref{lem:comparison}, the result of the iteration will
	be of the form $\bar{M}=J_{\alpha_{\bar{M}}}\left[D,\mathrm{CU}\mets S\right]$.
	Note that these will all be short iterable pre-mice (as iterations
	of short mice). Let $\alpha^{*}=\sup\left\{ \alpha_{\bar{M}}\mid M\in X\right\} $,
	$M^{*}=J_{\alpha^{*}}\left[D,\mathrm{CU}\mets S\right]$. We claim
	that $X\con M^{*}$: if $M\in X$, $i:M\to\bar{M}$ the iteration
	embedding, then $\bar{M}$ is an initial segment of $M^{*}$,  and
	$M\cong\mathcal{H}_{1}^{\bar{M}}\left(\rho_{M}\cup i(p_{M})\right)\in M^{*}$,
	so $M\in M^{*}$.

	Assume first that for all $M\in X$, $\alpha^{*}>\alpha_{\bar{M}}$
	and let $M^{**}$ be the transitive collapse of $\mathcal{H}_{1}^{M^{*}}(\lambda^{+})$.
	This is a short $D$-mouse, and we claim that already $X\con M^{**}$.
	Let $M\in X$, then $M=\mathcal{H}_{1}^{\bar{M}}\left(a\right)$ for
	some $a\in\bar{M}$. Since $\alpha^{*}>\alpha_{\bar{M}}$, $\bar{M}\in M^{*}$,
	so $M^{*}$ satisfies the $\Sigma_{1}$ statement ``$\exists y\exists z\left(M=\mathcal{H}_{1}^{y}\left(z\right)\right)$'',
	so such $y,z$ will be in $\mathcal{H}_{1}^{M^{*}}(\lambda^{+})$
	($\left|M\right|<\lambda$ so it can be coded as a subset of $\lambda^{+}$),
	so $M\in M^{**}$. So $x$ can be constructed in $M^{**}$, which
	is \emph{a mouse} with minimal measure above $\left|\mathrm{tr}(x)\right|$
	so $x$ is in $\lp(M^{**})$.

	Now assume $\alpha^{*}=\alpha_{\bar{N}}$ for some $N\in X$. This
	means that there is some $N\in X$ such that $M\leq_{D}N$ for every
	$M\in X$. Hence by lemma \ref{lem:mice order}, $X\con L\left[N\right]$
	so $x\in L\left[N\right]$. Assume $x\in J_{\kappa}\left[N\right]$
	and iterate $N$ to some $\bar{N}$ with $\min\meas(\bar{N})>\lambda'=\max\left\{ \lambda,\kappa\right\} ^{+}$.
	Then $x\in M^{**}:=\mathcal{H}_{1}^{\bar{N}}(\lambda')$ which is
	a mouse with minimal measure $>\left|\mathrm{tr}(X)\right|$.
\end{proof}

\section{\label{sec:Short-mice-C-star}Short mice in $C^{*}$}
\begin{thm}
	\label{thm:mice in C*}Let $D\in C^{*}$ be simple . Then any short
	$D$-mouse is in $C^{*}$. As a result, $\left(K\left[D\right]\right)^{C^{*}}=K\left[D\right]$.
\end{thm}

\begin{proof}
	Assume towards contradiction this is not the case, and let $M=\left\langle \left|M\right|,D,F\right\rangle $
	be the $<_{D}$-minimal short $D$-mouse not in $C^{*}$. Our goal
	is to find an appropriate sequence $\vec{\theta}=\left\langle \theta_{\gamma}\mid\gamma<\chi\right\rangle $
	where $\chi:=\otp\dom(F)$, iterate $M$ according to $\vec{\theta}$
	and show, following the proof of \cite[theorem 5.5]{IMEL}, that in
	fact $M^{\chi}\in C^{*}$, and from that deduce that $M\in C^{*}$,
	by contradiction.

	For any sequence $\left\langle \theta_{\gamma}\mid\gamma<\chi\right\rangle $,
	the final model in the iteration is $M^{\chi}=J_{\alpha}\left[D,F_{\chi}\right]$
	for some simple $F_{\chi}$, where by the construction $F_{\chi}=j^{\chi}(F)$
	and $\otp\dom F_{\chi}=\chi$. For any $\gamma<\chi$, the $\gamma$th
	element of $F_{\chi}$ is $\kappa_{\theta_{\gamma}}^{\gamma}$, and
	the corresponding $M^{\chi}$-ultrafilter $U_{\gamma}^{\chi}$ satisfies
	that for $X\in M^{\chi}\cap\power(\kappa_{\theta_{\gamma}}^{\gamma})$,
	$X\in U_{\gamma}^{\chi}$ iff $X$ contains a final segment of $E_{\gamma}:=\left\{ \kappa_{\beta}^{\gamma}\mid\beta<\theta_{\gamma}\right\} $
	iff $X$ contains a cofinal subset of $E_{\gamma}$. So for any cofinal
	$\bar{E}\con E_{\gamma}$, if $F_{\bar{E}}$ is the filter generated
	by final segments of $\bar{E}$, then $U_{\gamma}^{\chi}=F_{\bar{E}}\cap M^{\chi}$.
	So we want to find a sequence $\vec{\theta}$ for which we can show
	the following claim which will imply that $M^{\chi}\in C^{*}$:
	\begin{claim}
		\label{claim:cofinal sequences}There is a sequence $\left\langle \bar{E}_{\gamma}\mid\gamma<\chi\right\rangle \in C^{*}$
		such that for every $\gamma$ $\bar{E}_{\gamma}$ is cofinal in $E_{\gamma}$.
	\end{claim}

	To choose the right $\vec{\theta}$ and prove the claim, we need to
	be able to identify in $C^{*}$ unboundedly many critical points of
	the iteration. For this, we need the following crucial claims:
	\begin{claim}
		\label{claim:characterize-regulars}For every $\gamma<\chi$ and $\beta$,
		if $\eta$ is such that $\kappa_{0}^{\gamma}<\eta<\kappa_{\beta}^{\gamma}$
		and $M_{\beta}^{\gamma}\vDash\eta$ is regular, then either there
		is $\alpha<\beta$ such that $\eta=\kappa_{\alpha}^{\gamma}$ or $\mathrm{cf}^{V}(\eta)=\delta_{M}$.
	\end{claim}

	\begin{claim}
		\label{claim:regulars-agree}For every $\gamma<\chi$ and $\beta\leq\theta_{\gamma}$,
		$\eta\leq\kappa_{\beta}^{\gamma}$ is regular in $M_{\beta}^{\gamma}$
		iff it is regular in $K^{*}:=\left(K\left[D\right]\right)^{C^{*}}$.
		In particular every $\kappa_{\beta}^{\gamma}$ is regular in $K^{*}$.
	\end{claim}

	Let's first assume these claims and show that claim \ref{claim:cofinal sequences}
	(and then the theorem) follows. First we need to find the appropriate
	$\left\langle \theta_{\gamma}\mid\gamma<\chi\right\rangle $. We define
	in $C^{*}$ a sequence $\left\langle \lambda_{\gamma}\mid\gamma<\chi\right\rangle $
	recursively such that $\lambda_{0}>\left|M\right|^{V}$ and for every
	$\gamma<\chi$:
	\begin{enumerate}
		\item $C^{*}\vDash\lambda_{\gamma}$ is a regular cardinal;
		\item $\gamma'<\gamma$ implies $\lambda_{\gamma'}<\lambda_{\gamma}$;
		\item $\forall\beta<\lambda_{\gamma}$ $\exists\eta\in\left(\beta,\lambda_{\gamma}\right)$
		      such that $K^{*}\vDash\eta$ is a regular cardinal, and $\mathrm{cf}^{V}(\eta)=\omega\leftrightarrow\delta_{M}\ne\omega$.
	\end{enumerate}
	Assume we've defined $\left\langle \lambda_{\gamma'}\mid\gamma'<\gamma\right\rangle \in C^{*}$
	for some $\gamma<\chi$. If $\lambda>\sup\left\langle \lambda_{\gamma'}\mid\gamma'<\gamma\right\rangle $
	is regular in $V$ then it satisfies 1-3: regularity in $V$ implies
	regularity in $C^{*}$, and 2 is clear. To show 3 iterate $M$ according
	to $\left\langle \lambda_{\gamma'}\mid\gamma'<\gamma\right\rangle ^{\frown}\left\langle \lambda\right\rangle $
	(where in the $\gamma$th stages we iterate up to $\lambda$), and
	consider the iteration points $\left\langle \kappa_{\beta}^{\gamma}\mid\beta<\lambda\right\rangle $.
	Since $\lambda$ is regular in $V$ and we can assume $>\omega_{1}$,
	there are cofinally many iteration points of cofinality $\omega$
	and cofinally many of cofinality $>\omega$. Each such iteration point
	is regular in $K^{*}$ by claim \ref{claim:regulars-agree} (to be
	proved below), so $\lambda$ satisfies 3 (choosing the appropriate
	iteration points according to $\delta_{M}$). So there are ordinals
	satisfying 1-3, and whether an ordinal satisfies 1-3 can be determined
	within $C^{*}$. So in $C^{*}$ we can define $\lambda_{\gamma}$
	to be the first ordinal $>\sup\left\langle \lambda_{\gamma'}\mid\gamma'<\gamma\right\rangle $
	(or $>\left|M\right|$ if $\gamma=0$) satisfying 1-3. Note that this
	defines the sequence $\left\langle \lambda_{\gamma}\mid\gamma<\chi\right\rangle $
	recursively in $C^{*}$ so we can in fact proceed through limit stages.

	Consider now $\lambda=\lambda_{\gamma}>\sup\left\langle \lambda_{\gamma'}\mid\gamma'<\gamma\right\rangle $
	which satisfies 1-3, and iterate $M$ only according to $\kappa(\gamma)$
	such that the image of $\kappa(\gamma)$ is greater than $\lambda$.
	By 3 $\lambda$ is the limit of ordinals $\eta$ satisfying $K^{*}\vDash\eta$
	is a regular cardinal, and $\mathrm{cf}^{V}(\eta)=\omega\leftrightarrow\delta_{M}\ne\omega$.
	By claims \ref{claim:characterize-regulars} and \ref{claim:regulars-agree}
	each such $\eta$ is an iteration point, so also $\lambda$ must be
	an iteration point. So we have some iteration $j:M\to N$ by the ultrafilter
	$F(\gamma)$ of some length $\theta$, such that $j(\kappa(\gamma))=\lambda$.
	Note that everything below $\kappa(\gamma)$ is fixed, so $j(F)\mets\lambda=F\mets\kappa(\gamma)$,
	and we can iterate $N$ according to $\left\langle \lambda_{\gamma'}\mid\gamma'<\gamma\right\rangle $.
	Since in $N$ $\lambda$ is strongly inaccessible and is above $\sup\left\langle \lambda_{\gamma'}\mid\gamma'<\gamma\right\rangle $,
	it is not moved by this iteration. It is a known fact that the resulting
	iteration is isomorphic to the iteration of $M$ according to $\left\langle \lambda_{\gamma'}\mid\gamma'<\gamma\right\rangle ^{\frown}\left\langle \theta\right\rangle $.
	So we get that in fact $\lambda$ is the image of $\kappa(\gamma)$
	under the iteration according to $\left\langle \lambda_{\gamma'}\mid\gamma'<\gamma\right\rangle ^{\frown}\left\langle \theta\right\rangle $,
	i.e. $\lambda=\kappa_{\theta}^{\gamma}$. We set $\theta_{\gamma}$
	to be this $\theta$.

	So to conclude, we've defined recursively $\left\langle \lambda_{\gamma}\mid\gamma<\chi\right\rangle $
	inside $C^{*}$, and we have (in $V$!) a sequence $\vec{\theta}=\left\langle \theta_{\gamma}\mid\gamma<\chi\right\rangle $
	such that when iterating $M$ according to $\vec{\theta}$ we have
	for every $\gamma$, $\lambda_{\gamma}=\kappa_{\theta_{\gamma}}^{\gamma}$
	satisfying the requirements.
	\begin{proof}[Proof of claim \ref{claim:cofinal sequences}]
		\renewcommand{\qedsymbol}{$\square$\footnotesize{}{Claim \ref{claim:cofinal sequences}}} We
		iterate $M$ according to $\vec{\theta}$ as above and obtain (in
		$V$) the iteration points $E_{\gamma}=\left\{ \kappa_{\beta}^{\gamma}\mid\beta<\theta_{\gamma}\right\} $.
		For every $\gamma<\chi$ set
		\[
			\bar{E}_{\gamma}=\begin{cases}
				\left\{ \eta<\lambda_{\gamma}\mid\eta\text{ is regular in \ensuremath{K^{*}} and }\mathrm{cf}^{V}\left(\eta\right)=\omega\right\} & \delta_{M}\ne\omega \\
				\left\{ \eta<\lambda_{\gamma}\mid\eta\text{ is regular in \ensuremath{K^{*}} }\mathrm{cf}^{V}\left(\eta\right)\ne\omega\right\}   & \delta_{M}=\omega
			\end{cases}
		\]
		Each of theses sets is in $C^{*}$, and by claims \ref{claim:characterize-regulars}
		and \ref{claim:regulars-agree} and lemma \ref{lem:fixed cof}, $\bar{E}_{\gamma}\con E_{\gamma}$
		for every $\gamma$. By our choice of $\lambda_{\gamma}$ and $\theta_{\gamma}$,
		they are cofinal in $\lambda_{\gamma}=\kappa_{\theta_{\gamma}}^{\gamma}$
		(by clause 3). Since $\left\langle \lambda_{\gamma}\mid\gamma<\chi\right\rangle \in C^{*}$,
		the sequence $\left\langle \bar{E}_{\gamma}\mid\gamma<\chi\right\rangle $
		is in $C^{*}$ as well.\qedhere
	\end{proof}
	Now as we claimed in the beginning of the proof, we can compute each
	ultrafilter $U_{\gamma}^{\chi}$ from the sequence $\bar{E}_{\gamma}$,
	hence in $C^{*}$ we can compute $M^{\chi}$. Thus by lemma \ref{lem:mice order}
	$M\in C^{*}$, by contradiction.\renewcommand{\qedsymbol}{$\square$\footnotesize{}{Theorem \ref{thm:mice in C*}}}
\end{proof}
We now move on to prove claims \ref{claim:characterize-regulars}
and \ref{claim:regulars-agree}.
\begin{proof}[Proof of claim \ref{claim:characterize-regulars}]
	By examining the proof of claim 2 in the proof of \cite[theorem 5.5]{IMEL},
	which corresponds to our claim, one can see that it can be stated
	as a general lemma:
	\begin{lem}[{in \cite[theorem 5.5]{IMEL}}]
		\label{lem:characterize-regulars}Assume that $\left\langle N_{\alpha},k_{\alpha,\beta}\mid\alpha\leq\beta<\theta\right\rangle $
		is an ultrapower iteration of~ $U_{0}\in N_{0}$ which is (in $N_{0}$)
		an ultrafilter on $\lambda_{0}$, i.e. for every $\alpha<\theta$,
		$N_{\alpha+1}=\ult\left(N_{\alpha},i_{0,\alpha}(U_{0})\right)$ and
		for limit $\alpha$ $N_{\alpha}$ is the directed limit. Let $\zeta_{0}=\left(\lambda_{0}^{+}\right)^{M_{0}}$(or
		$M_{0}\cap\ord$ if this doesn't exist) and $\delta=\mathrm{cf}^{V}(\zeta_{0})$.
		Set $\lambda_{\alpha}=i_{0,\alpha}(\lambda_{0})$ and assume that
		for every $\alpha<\theta$, $\left(\lambda_{\alpha}\right)^{\lambda_{\alpha}}\cap N_{\alpha}$
		is the increasing union of $\delta$ members of $N_{\alpha}$, each
		of cardinality $\lambda_{\alpha}$ in $N_{\alpha}$. Then for every
		$\beta<\theta$ and every $\eta$ such that $\lambda_{0}<\eta<\lambda_{\beta}$,
		if $N_{\beta}\vDash\eta$ is regular, then either $\exists\gamma<\beta$
		such that $\eta=\lambda_{\gamma}$ or $\mathrm{cf}^{V}(\eta)=\delta$.
	\end{lem}

	Recall that $\mathrm{cf}^{V}(\xi(\gamma))=\delta_{M}=\mathrm{cf}^{V}(\alpha_{M})$,
	so to use this lemma, we need to show that for every $\gamma$ and
	$\beta<\lambda_{\gamma}$, $\big(\kappa_{\beta}^{\gamma}\big)^{\kappa_{\beta}^{\gamma}}\cap M_{\beta}=\bigcup_{\psi<\delta_{M}}F_{\psi}$
	such that $F_{\psi}\in M_{\beta}^{\gamma}$ and $\left|F_{\psi}\right|=\kappa_{\beta}^{\gamma}$.
	Set for the moment $N=M_{\beta}^{\gamma}$, $\kappa=\kappa_{\beta}^{\gamma}$,
	$\xi=\xi_{\beta}^{\gamma}=\left(\kappa^{+}\right)^{N}$. Assume $N=J_{\alpha_{N}}\left[D,F\right]$
	and note that the image of the iteration embedding from $M$ to $N$
	is cofinal, so $\mathrm{cf}^{V}(\alpha_{N})=\mathrm{cf}^{V}(\alpha_{M})$.
	Since $N=\mathcal{H}_{1}^{N}(\rho_{N}\cup p_{N})$, every $h\in\kappa^{\kappa}\cap N$
	is in some $\mathcal{H}_{1}^{J_{\beta}\left[D,F\right]}(\rho_{N}\cup p_{N})$
	for some $\beta<\alpha_{N}$ (as in the proof of lemma \ref{lem:fixed cof}),
	and since $\rho_{M}<\kappa$, this is of size at most $\kappa$. So
	if we take $\left\langle \beta_{\psi}\mid\psi<\delta_{M}\right\rangle $
	cofinal in $\alpha_{N}$, then $F_{\psi}=\kappa^{\kappa}\cap\mathcal{H}_{1}^{J_{\beta_{\psi}}\left[D,F\right]}(\rho_{N}\cup p_{N})$
	is as required. So we can apply lemma \ref{lem:characterize-regulars}
	and get our claim. \renewcommand{\qedsymbol}{$\square$\footnotesize{}{Claim \ref{claim:characterize-regulars}}}
\end{proof}
As for claim \ref{claim:regulars-agree}, it follows by the choice
of $M$:
\begin{proof}[Proof of claim \ref{claim:regulars-agree}]
	Fix $\gamma<\chi$ and $\beta\leq\theta_{\gamma}$, and we need
	to show that $\eta\leq\kappa_{\beta}^{\gamma}$ is regular in $M_{\beta}^{\gamma}$
	iff it is regular in $K^{*}$. This will follow by showing that $\power(\kappa_{\beta}^{\gamma})\cap K^{*}=\power(\kappa_{\beta}^{\gamma})\cap M_{\beta}^{\gamma}$.

	If $X\in\power(\kappa_{\beta}^{\gamma})\cap K^{*}$, by theorem $\ref{thm:K(D)}$
	$X\in\mathrm{lp}(N)$ for some short $D$-mouse $N\in K^{*}$. If
	$M_{\beta}^{\gamma}\leq_{D}N\in K^{*}$ then in particular we'd get
	$M\in K^{*}$ by contradiction to the choice of $M$, so we must have
	$N<_{D}M_{\beta}^{\gamma}$. So there are iterates $J_{\delta}\left[D,F^{*}\right]$
	and $J_{\delta'}\left[D,F^{*}\right]$ of $N,M_{\beta}^{\gamma}$
	respectively, such that $\delta\leq\delta'$. $X\in\mathrm{lp}(N)$
	so $X$ is fixed in the iteration, hence $X\in J_{\delta}\left[D,F^{*}\right]\con J_{\delta'}\left[D,F^{*}\right]$.
	We can also assume that the iteration of $M_{\beta}^{\gamma}$ is
	done with measurables $\geq\kappa_{\beta}^{\gamma}$ so $\power(\kappa_{\beta}^{\gamma})\cap M_{\beta}^{\gamma}=\power(\kappa_{\beta}^{\gamma})\cap J_{\delta'}\left[D,F^{*}\right]$
	so $X\in M_{\beta}^{\gamma}$.

	If $X\in\power(\kappa_{\beta}^{\gamma})\cap M_{\beta}^{\gamma}$,
	$M_{\beta}^{\gamma}=J_{\tilde{\alpha}}[D,\tilde{F}]$, then there
	is some $\eta\in M_{\beta}^{\gamma}$ such that $X$ is definable
	in $J_{\eta}[D,\tilde{F}]$ from a finite set of parameters $p$,
	so $X\in\mathcal{H}_{1}^{M_{\beta}^{\gamma}}\left(\kappa_{\beta}^{\gamma}\cup p\right)$
	i.e. $X$ is in a mouse $N$ which is smaller than $M_{0}^{0}$, but
	$M_{0}^{0}$ is the minimal mouse not in $K^{*}$, so $N$ is in $K^{*}$,
	hence $X\in\power(\kappa_{\beta}^{\gamma})\cap K^{*}$. \renewcommand{\qedsymbol}{$\square$\footnotesize{}{Claim \ref{claim:regulars-agree}}}
\end{proof}

\section{\label{sec:The-C-star-of}The $C^{*}$ of a short $L\left[\mathcal{U}\right]$ }

Our next goal is to generalize theorem 5.16 of \cite{IMEL} which
states that if $V=L^{\mu}$ then $C^{*}=M_{\omega^{2}}\left[E\right]$
where $M_{\omega^{2}}$ is the $\omega^{2}$th iterate of $V$ and
$E=\left\{ \kappa_{\omega\cdot n}\mid1\leq n<\omega\right\} $ is
the sequence of the $\omega\cdot n$ iteration points. As noted in
\cite{IMEL}, this can be generalized using the same proof to a model
with finitely many measurable cardinals $\left\{ \kappa^{0}\til\kappa^{m}\right\} $,
using the Prikry sequences $E^{i}=\left\{ \kappa_{\omega\cdot n}^{i}\mid1\leq n<\omega\right\} $
for $i\leq m$. However, if we want infinitely many measurables, $\left\langle \kappa^{i}\mid i<\omega\right\rangle $,
we'd need to show that the \emph{sequence} $\left\langle E^{i}\mid i<\omega\right\rangle $
is in $C^{*}$, and not only every $E^{i}$, and at first sight it
is not clear that this is the case -- note for example that by \cite[theorem 5.10]{IMEL},
the sequence of measurables $\left\langle \kappa^{i}\mid i<\omega\right\rangle $
is necessarily \emph{not} in $C^{*}$. We will however show that using
the previous theorem, we can construct the sequence $\left\langle E^{i}\mid i<\omega\right\rangle $,
and even longer (though still short) sequences, in $C^{*}$.
\begin{thm}
	\label{thm:C-star-in-L(U)} Assume $V=L\left[\mathcal{U}\right]$
	where $\mathcal{U}=\left\langle U^{\gamma}\mid\gamma<\chi\right\rangle $
	is a short sequence of measures on the increasing measurables $\left\langle \kappa^{\gamma}\mid\gamma<\chi\right\rangle $
	($\chi<\kappa^{0}$). Iterate $V$ according to $\mathcal{U}$ where
	each measurable is iterated $\omega^{2}$ many times, to obtain $\left\langle M_{\alpha}^{\gamma},\mid\gamma<\chi,\alpha\leq\omega^{2}\right\rangle $,
	with iteration points $\left\langle \kappa_{\alpha}^{\gamma}\mid\gamma<\chi,\alpha\leq\omega^{2}\right\rangle $
	and set $M^{\chi}$ as the directed limit of this iteration. Then
	\[
		C^{*}=M^{\chi}\left[\left\langle E^{\gamma}\mid\gamma<\chi\right\rangle \right]
	\]
	where $E^{\gamma}=\left\langle \kappa_{\omega\cdot n}^{\gamma}\mid1\leq n<\omega\right\rangle $.
\end{thm}

\begin{rem}
	This theorem appears already in a paper by Welch, \cite[Theorem 6.1]{welch2022closed},
	where the proof only states that since there are no measurable limits
	of measurables, ``those arguments in \cite{IMEL} can, with some
	additional work, be used to get the result''. However, as we noted
	above, dealing with each measurable separately is \emph{prima facie}
	not sufficient to show that the \emph{sequence} of Prikry sequences
	$\left\langle E^{\gamma}\mid\gamma<\chi\right\rangle $ is in $C^{*}$,
	which seems crucial for the argument. Hence we include a detailed
	proof which addresses this issue, using the methods developed above.
	In the forthcoming \cite{welch2022C*} Welch gives a detailed proof
	of a more general result (see below), based on the techniques and
	results of \cite{welch2022closed}, however our proof here was obtained
	independently and does not rely on these results, and is in some sense
	more elementary, as Welch's argument requires more advanced inner
	model theory.
\end{rem}

\begin{proof}
	First we prove that $\left\langle E^{\gamma}\mid\gamma<\chi\right\rangle \in C^{*}$
	by giving a recursive definition for the sequence inside $C^{*}$.
	Denote for every $\gamma$ the filter on $\kappa_{\omega^{2}}^{\gamma}$
	generated by end-segments of $E^{\gamma}$ by $\bar{U}^{\gamma}$
	(recall that this is also the image of $U^{\gamma}$ under the corresponding
	embedding). Assume we've shown that $\left\langle E^{\gamma}\mid\gamma<\eta\right\rangle \in C^{*}$
	for some $\eta<\chi$, so also $D^{\eta}:=\left\langle \bar{U}^{\gamma}\mid\gamma<\eta\right\rangle \in C^{*}$,
	hence by theorem \ref{thm:mice in C*} $K\left[D^{\eta}\right]=\left(K\left[D^{\eta}\right]\right)^{C^{*}}$.
	Denote $K^{\eta}:=K\left[D^{\eta}\right]$.
	\begin{claim}
		\label{claim:equal-power-set}For every $\alpha$ , $\power(\kappa_{\alpha}^{\eta})\cap K^{\eta}=\power(\kappa_{\alpha}^{\eta})\cap M_{\alpha}^{\eta}$.
		In particular, for every $\delta\leq\kappa_{\alpha}^{\eta}$ $K^{\eta}\vDash$``$\delta$
		is regular'' iff $M_{\alpha}^{\eta}\vDash$``$\delta$ is regular''.
	\end{claim}

	\begin{proof}
		Denote the embedding $j^{\eta}\circ i_{0,\alpha}^{\eta}:V\to M_{\alpha}^{\eta}$
		by $j$. So $M_{\alpha}^{\eta}=L\left[j(\mathcal{U})\right]$ and
		since $\chi<\min\mathcal{U}$, this is $L\left[\left\langle j(U^{\gamma})\mid\gamma<\chi\right\rangle \right]$.
		For $\gamma<\eta$, $j(U^{\gamma})=\bar{U}^{\gamma}$ and $j(U^{\eta})=U_{\alpha}^{\eta}$.
		We claim that for $\gamma>\eta$ $j(U^{\gamma})=U^{\gamma}\cap M_{\alpha}^{\eta}$.
		Fix $\gamma>\eta$. First, since we are only dealing here with iterations
		of length $\leq\omega^{2}$ for each measure, $j(\kappa^{\gamma})=\kappa^{\gamma}$.
		Every member of $M_{\alpha}^{\eta}$ is represented by some $f:\kappa_{i_{1}}\times\dots\times\kappa_{i_{n}}\to V$
		for some $\kappa_{i_{1}},\dots,\kappa_{i_{n}}$ which are among the
		iteration points of $j$. If $\left[f\right]\in j(U^{\gamma})$, then
		without loss of generality $f(\gamma_{1}\til\gamma_{n})\in U^{\gamma}$
		for all $\left(\gamma_{1}\til\gamma_{n}\right)\in\kappa_{i_{1}}\times\dots\times\kappa_{i_{n}}$.
		Let
		\[
			A=\bigcap\left\{ f(\gamma_{1}\til\gamma_{n})\mid\left(\gamma_{1}\til\gamma_{n}\right)\in\kappa_{i_{1}}\times\dots\times\kappa_{i_{n}}\right\} .
		\]
		$A\in U^{\gamma}$ since $\left|\kappa_{i_{1}}\times\dots\times\kappa_{i_{n}}\right|<\kappa^{\gamma}$,
		and we get that $j(A)\con\left[f\right]$. If $A\in U^{\gamma}$ then
		also $A'=\left\{ \xi\in A\mid\xi\,\text{is inaccessible}>\sup\left\{ \kappa^{\gamma}\mid\gamma<\eta\right\} \right\} \in U^{\gamma}$,
		but for every $\xi\in A'$ $\xi=j(\xi)\in j(A)$ so $A'\con j(A)$
		hence $j(A)\in U^{\gamma}$. So every member of $j(U^{\gamma})$ contains
		a member of $U^{\gamma}$, so this is exactly the ultrafilter generated
		by $U^{\gamma}\cap M_{\alpha}^{\eta}$. So to conclude, we have $M_{\alpha}^{\eta}=L\left[D^{\eta},U_{\alpha}^{\eta},\left\langle U^{\gamma}\mid\eta<\gamma<\chi\right\rangle \right]$.

		If $x\in\power(\kappa_{\alpha}^{\eta})\cap M_{\alpha}^{\eta}$, then
		$x\in J_{\theta}\left[D^{\eta},U_{\alpha}^{\eta},\left\langle U^{\gamma}\mid\eta<\gamma<\chi\right\rangle \right]$
		for some $\theta$, so taking some $\rho>\left\langle \sup2^{\kappa_{\omega^{2}}^{\gamma}}\mid\gamma<\eta\right\rangle $
		and $<\kappa^{\eta}$, we can find parameters $p$ such that $x\in\mathcal{H}_{1}^{M_{\alpha}^{\eta}}\left(\rho\cup p\right)$
		which is a $D^{\eta}$-mouse, hence $x\in K^{\eta}$.

		Let $x\in\power(\kappa_{\alpha}^{\eta})\cap K^{\eta}$ and let $P$
		be a $D^{\eta}$-mouse such that $x\in\lp(P)$, $P=J_{\mu}\left[D^{\eta},F\right]$
		where $\min F>\kappa_{\alpha}^{\eta}$. Note that the order type of
		the sequence $D^{\eta}\cup F$ is at most $\chi$, since a mouse with
		more than $\chi$ many measures will generate a sharp for a canonical
		model with $\chi$ many measures, hence it cannot exist in $L\left[\mathcal{U}\right]$.
		So we can assume (after some reindexing) that for some $\chi'\leq\chi$,
		$F=\left\langle F^{\gamma}\mid\eta\leq\gamma<\chi'\right\rangle $.
		Choose some increasing sequence of regular cardinals $S=\left\langle \lambda^{\gamma}\mid\eta\leq\gamma<\chi'\right\rangle $
		such that $\lambda^{\eta}$ is greater than the power set of each
		ordinal measured in $F$ or $\mathcal{U}$. Iterating $P$ and $M_{\alpha}^{\eta}$
		according to $S$, using the measures in $\left\langle F^{\gamma}\mid\eta\leq\gamma<\chi'\right\rangle $
		and $\left\langle U_{\alpha}^{\eta}\right\rangle ^{\frown}\left\langle U^{\gamma}\mid\eta<\gamma<\chi'\right\rangle $
		respectively, will result in models
		\begin{align*}
			\bar{P} & =J_{\theta}\left[D^{\eta},\mathrm{CU}\mets S\right]\text{ and }                                     \\
			\bar{M} & =L\left[D^{\eta},\mathrm{CU}\mets S,\langle\tilde{U}^{\gamma}\mid\chi'\leq\gamma<\chi\rangle\right]
		\end{align*}
		(c.f. lemma \ref{lem:comparison}, $\mathrm{CU}$ is the closed unbounded
		simple predicate, $\tilde{U}^{\gamma}$ the image of $U^{\gamma}$
		under the iteration) and in particular, $\bar{P}$ is an initial segment
		of $\bar{M}$.  The measures used in the iterations are on ordinals
		$\geq\kappa_{\alpha}^{\eta}$, so the power set of $\kappa_{\alpha}^{\eta}$
		is preserved throughout the iteration, i.e.
		\[
			\power(\kappa_{\alpha}^{\eta})\cap P=\power(\kappa_{\alpha}^{\eta})\cap\bar{P}\,\,\,\text{and}\,\,\,\power(\kappa_{\alpha}^{\eta})\cap M_{\alpha}^{\eta}=\power(\kappa_{\alpha}^{\eta})\cap\bar{M}.
		\]
		and in particular $x\in\bar{P}\cap\power(\kappa_{\alpha}^{\eta})\con\bar{M}\cap\power(\kappa_{\alpha}^{\eta})$,
		hence $x\in M_{\alpha}^{\eta}$.\renewcommand{\qedsymbol}{$\square$\footnotesize{}{Claim \ref{claim:equal-power-set}}}
	\end{proof}
	\begin{claim}
		\label{claim:small-regular}If $\delta<\kappa^{\eta}$ is such that
		$M_{0}^{\eta}\vDash\delta$ is regular, then $\delta$ is not of cofinality
		$\omega$ in $V$.
	\end{claim}

	\begin{proof}
		$\delta$ is represented in $M_{0}^{\eta}$ by some function $f:\kappa_{i_{1}}\times\dots\times\kappa_{i_{n}}\to V$
		for some $\kappa_{i_{1}},\dots,\kappa_{i_{n}}$ which are among the
		iteration points of $j^{\eta}$. Since $\delta$ is in $M_{0}^{\eta}$
		a regular cardinal $<\kappa^{\eta}$, we can assume that the range
		of $f$ consists of regular cardinals $<\kappa^{\eta}$. Let $E=\kappa_{i_{1}}\times\dots\times\kappa_{i_{n}}$
		and let $U^{E}\in V$ be the ultrafilter on $E$ given by the iteration,
		so that $\delta$ is represented by $\left[f\right]_{U^{E}}$. Let
		$\lambda_{*}$ be the supremum modulo $U^{E}$ of $\rng(f)$. If there
		is a $U^{E}$-large set on which $f$ is constant with value $\lambda$
		then we have $\delta=j^{\eta}(\lambda)$ and in particular is not
		of cofinality $\omega$ in $V$. Otherwise, we claim that $\mathrm{cf}^{V}(\delta)=\lambda_{*}^{+}$
		(and in particular $>\omega$). Using GCH in $V$, enumerate $\text{\ensuremath{\prod_{X\in E}f(X)}}$
		by $\left\langle g_{\alpha}\mid\alpha<\lambda_{*}^{+}\right\rangle $.
		Let $I$ be the ideal of bounded sets on $E$ ($A\in I$ if at each
		coordinate $k=1\til n$ $A$ is bounded below $\kappa_{i_{k}}$).
		We define by induction a sequence $\left\langle f_{\alpha}\mid\alpha<\lambda_{*}^{+}\right\rangle \in\text{\ensuremath{\prod_{X\in E}f(X)}}$,
		increasing modulo $I$ s.t. for every $\alpha$, $g_{\alpha}\leq_{I}f_{\alpha}$.
		Start with $f_{0}=g_{0}$ and at successor step take the pointwise
		maximum of $f_{\alpha}$ and $g_{\alpha+1}$. At limit stage $\alpha<\lambda_{*}^{+}$,
		since we assumed that $f$ is not constant on any large set, and all
		its values are regular cardinals, the set of $A_{\alpha}=\left\{ X\in E\mid\mathrm{cf}(f(X))=\mathrm{cf}(\alpha)\right\} $
		is bounded, and on every $X\in E\smin A_{\alpha}$, the sequence $\left\langle f_{\gamma}(X)\mid\gamma<\alpha\right\rangle $
		is bounded below $f(X)$ (by the mismatch in cofinality), so we can
		set $f_{\alpha}(X)=\max\left\{ \sup\left\langle f_{\gamma}(X)\mid\gamma<\alpha\right\rangle ,g_{\alpha}(X)\right\} $
		on every $X\in E\smin A_{\alpha}$. The sequence $\left\langle f_{\alpha}\mid\alpha<\lambda_{*}^{+}\right\rangle $
		is increasing modulo $I$, and hence also modulo $U^{E}$, and is
		cofinal in $\text{\ensuremath{\prod_{X\in E}f(X)}}$ modulo $I$,
		and hence $\left[f_{\alpha}\right]_{U^{E}}$ is a cofinal sequence
		in $\left[f\right]_{U^{E}}=\delta$ of order type $\lambda_{*}^{+}$,
		as required. \renewcommand{\qedsymbol}{$\square$\footnotesize{}{Claim \ref{claim:small-regular}}}
	\end{proof}
	\begin{claim}
		\label{claim:iteration-points}$E^{\eta}=\left\langle \kappa_{\omega\cdot n}^{\eta}\mid1\leq n<\omega\right\rangle $
		is the sequence of the first $\omega$ ordinals greater than $\Upsilon^{\eta}:=\sup\left\{ \sup E^{\gamma}\mid\gamma<\eta\right\} $
		which are regular in $K^{\eta}$ but of cofinality $\omega$ in $V$.
	\end{claim}

	\begin{proof}
		For every $\eta$ and $\alpha$, $\kappa_{\alpha}^{\eta}$ is regular
		in $M_{\alpha}^{\eta}$, so by claim \ref{claim:equal-power-set}
		it is regular in $K^{\eta}$ as well, and for every $n\geq1$ $\kappa_{\omega\cdot n}^{\eta}=\sup\left\{ \kappa_{\omega(n-1)+k}^{\eta}\mid k<\omega\right\} $
		so is of cofinality $\omega$ in $V$.

		Now note that for every $\gamma<\eta$, $\sup E^{\gamma}=\kappa_{\omega^{2}}^{\gamma}<\kappa^{\eta}$
		(by the properties of iterations), so also $\Upsilon^{\eta}=\sup\left\{ \sup E^{\gamma}\mid\gamma<\eta\right\} <\kappa^{\eta}=\kappa_{0}^{\eta}$.
		Hence by claim \ref{claim:equal-power-set} every $\delta\in\left(\Upsilon^{\eta},\kappa^{\eta}\right)$
		is regular in $K^{\eta}$ iff it is regular in $M_{0}^{\eta}$, and
		by claim \ref{claim:small-regular} $\delta$ is not of cofinality
		$\omega$ in $V$. So in the interval $\left(\Upsilon^{\eta},\kappa^{\eta}\right]$
		there are no ordinals which are regular in $K^{\eta}$ but of cofinality
		$\omega$ in $V$. For $\delta\in\big(\kappa_{\omega\cdot n}^{\eta},\kappa_{\omega\cdot\left(n+1\right)}^{\eta}\big)$
		($n\geq0$), by claim \ref{claim:equal-power-set} we get that $\delta$
		is regular in $K^{\eta}$ iff it is regular in $M_{\omega\cdot(n+1)}^{\eta}$,
		in which case we get by the same proof as that of claim \ref{claim:characterize-regulars}
		that either $\delta=\kappa_{\omega\cdot n+k}^{\eta}$ for some $k>0$,
		or $\mathrm{cf}^{V}(\delta)=\left(\kappa^{\eta}\right)^{+}$, and
		in either case $\delta$ is not of cofinality $\omega$ in $V$.

		Combining these three observations we get that indeed the first $\omega$
		ordinals $>\Upsilon^{\eta}$ which are regular in $K^{\eta}$ but
		of cofinality $\omega$ in $V$ are $\left\langle \kappa_{\omega\cdot n}^{\eta}\mid1\leq n<\omega\right\rangle .$\renewcommand{\qedsymbol}{$\square$\footnotesize{}{Claim \ref{claim:iteration-points}}}
	\end{proof}
	This claim allows us to define $E^{\eta}$ in $C^{*}$ using the sequence
	$\left\langle E^{\gamma}\mid\gamma<\eta\right\rangle $. Hence the
	sequence $\left\langle E^{\gamma}\mid\gamma<\chi\right\rangle $ can
	be defined recursively inside $C^{*}$: if $\left\langle E^{\gamma}\mid\gamma<\eta\right\rangle \in C^{*}$
	then also $\left\langle E^{\gamma}\mid\gamma\leq\eta\right\rangle \in C^{*}$,
	and if $\eta\leq\chi$ is limit, and $\left\langle E^{\gamma}\mid\gamma<\delta\right\rangle $
	is defined recursively in $C^{*}$ for every $\delta<\eta$, then
	the recursion actually implies that $\left\langle \left\langle E^{\gamma}\mid\gamma<\delta\right\rangle \mid\delta<\eta\right\rangle \in C^{*}$,
	so also $\left\langle E^{\gamma}\mid\gamma<\eta\right\rangle \in C^{*}$.
	In particular, $\left\langle E^{\gamma}\mid\gamma<\chi\right\rangle \in C^{*}$.
	Now we claim that indeed $C^{*}=M^{\chi}\left[\left\langle E^{\gamma}\mid\gamma<\chi\right\rangle \right]$.
	First of all, if we again denote by $\bar{U}^{\gamma}$ the filter
	on $\kappa_{\omega^{2}}^{\gamma}$ generated by end-segments of $E^{\gamma}$,
	then $\left\langle \bar{U}^{\gamma}\mid\gamma<\chi\right\rangle \in C^{*}$
	and $M^{\chi}=L\left[\left\langle \bar{U}^{\gamma}\mid\gamma<\chi\right\rangle \right]$,
	so $M^{\chi}\left[\left\langle E^{\gamma}\mid\gamma<\chi\right\rangle \right]\con C^{*}$.
	On the other hand, the only ordinals which are of cofinality $\omega$
	in $V$ but not in $M^{\chi}$ are those whose cofinality in $M^{\chi}$
	is in $E^{\gamma}\cup\left\{ \sup E^{\gamma}\right\} $ for some $\gamma<\chi$,
	so in $M^{\chi}\left[\left\langle E^{\gamma}\mid\gamma<\chi\right\rangle \right]$
	we can correctly compute $C^{*}$ by replacing the quantifier $Q_{\omega}^{\mathrm{cf}}(\alpha)$
	with ``$\mathrm{cf}(\alpha)\in\left\{ \omega\right\} \cup\bigcup\left\{ E^{\gamma}\mid\gamma<\chi\right\} \cup\left\{ \sup E^{\gamma}\mid\gamma<\chi\right\} $'',
	so $C^{*}\con M^{\chi}\left[\left\langle E^{\gamma}\mid\gamma<\chi\right\rangle \right]$
	as required.\renewcommand{\qedsymbol}{$\square$\footnotesize{}{Theorem \ref{thm:C-star-in-L(U)}}}
\end{proof}
We now generalize the above theorem by calculating $C^{*}$ in an
extension of $L\left[\mathcal{U}\right]$ obtained by Prikry forcing
for some of the measurables of $\mathcal{U}$.
\begin{thm}
	\label{thm:C-star-in-Prikry-ext}Assume $V=L\left[\mathcal{U}\right]$
	where $\mathcal{U}$ is as above, $M^{\chi}$ and $E^{\gamma}$ as
	above . Let $A\in V$, $A\con\chi$, and set for $\gamma<\chi$
	\[
		\hat{E}^{\gamma}=\begin{cases}
			\left\langle \kappa^{\gamma}\right\rangle ^{\frown}E^{\gamma} & \gamma\in A    \\
			E^{\gamma}                                                    & \gamma\notin A
		\end{cases}=\begin{cases}
			\left\langle \kappa_{\omega\cdot n}^{\gamma}\mid0\leq n<\omega\right\rangle & \gamma\in A    \\
			\left\langle \kappa_{\omega\cdot n}^{\gamma}\mid1\leq n<\omega\right\rangle & \gamma\notin A
		\end{cases}
	\]
	and let $G$ be generic for the forcing over $V$ adding a Prikry
	sequence $G_{\alpha}\con\kappa^{\alpha}$ for every $\alpha\in A$.
	Then $\left(C^{*}\right)^{V\left[G\right]}=M^{\chi}\left[\left\langle \hat{E}^{\gamma}\mid\gamma<\chi\right\rangle \right].$
\end{thm}

\begin{rem}
	The difference between this case and the previous one is that by
	adding a Prikry sequence to $\kappa^{\gamma}$, it becomes of cofinality
	$\omega$, and in fact $C^{*}$ might not be able to distinguish it
	from the iteration points $\kappa_{\omega\cdot n}$ for $n\geq1$.
	The theorem does not imply that $A\in\left(C^{*}\right)^{V\left[G\right]}$
	-- it will be clear from the proof that $\left(C^{*}\right)^{V\left[G\right]}$
	might not distinguish which option was used in defining $\hat{E}^{\gamma}$,
	and as we noted earlier, $C^{*}$ doesn't contain any infinite sequence
	of the measurables, so it can't distinguish the ones that were collapsed
	from the ones that weren't.
\end{rem}

\begin{cor}
	If $A$ is finite then $\left(C^{*}\right)^{V\left[G\right]}=\left(C^{*}\right)^{V}$.
	If $A\in\left(C^{*}\right)^{V\left[G\right]}$ then $\left\langle \kappa^{\alpha}\mid\alpha\in A\right\rangle \in\left(C^{*}\right)^{V\left[G\right]}$
	and in particular, if $A\in\left(C^{*}\right)^{V\left[G\right]}$
	is infinite then $\left(C^{*}\right)^{V\left[G\right]}\ne\left(C^{*}\right)^{V}$.
\end{cor}

\begin{rem}
	In our paper \cite{yaar-iterating-Cstar} we proved this in the special
	case of a Prikry extension of $L^{\mu}$.
\end{rem}

\begin{proof}[Proof of Corollary]
	If $A$ is finite then $A\in\left(C^{*}\right)^{V}$ and the sequence
	$\langle\hat{E}^{\gamma}\mid\gamma<\chi\rangle$ can be computed from
	$A$ and $\left\langle E^{\gamma}\mid\gamma<\chi\right\rangle $ (recall
	$\left(C^{*}\right)^{V}=M^{\chi}\big[\langle E^{\gamma}\mid\gamma<\chi\rangle\big]$).
	If $A\in\left(C^{*}\right)^{V\left[G\right]}$ then $\left\langle \kappa_{\alpha}\mid\alpha\in A\right\rangle $
	can be computed from $\langle\hat{E}^{\gamma}\mid\gamma<\chi\rangle$
	so is in $\left(C^{*}\right)^{V\left[G\right]}$. As we noted, if
	$A$ is infinite then by \cite[theorem 5.10]{IMEL}, $\left\langle \kappa^{\alpha}\mid\alpha\in A\right\rangle \notin\left(C^{*}\right)^{V}$
	.
\end{proof}
\begin{proof}[Proof of Theorem \ref{thm:C-star-in-Prikry-ext}]
	We repeat the proof of theorem \ref{thm:C-star-in-L(U)}. For the
	$\supseteq$ direction, we note that the only difference is that in
	claim \ref{claim:iteration-points}, for $\eta\in A$, the sequence
	of the first $\omega$ ordinals greater than $\sup\left\{ \sup E^{\gamma}\mid\gamma<\eta\right\} $
	which are regular in $K^{\eta}$ but of cofinality $\omega$ in $V\left[G\right]$
	is now the sequence $\langle\kappa_{\omega\cdot n}^{\eta}\mid0\leq n<\omega\rangle$
	instead of $\langle\kappa_{\omega\cdot n}^{\gamma}\mid1\leq n<\omega\rangle$,
	since also $\kappa^{\eta}$ has cofinality $\omega$ in $V\left[G\right]$,
	but otherwise there is no difference in cofinality $\omega$ between
	$V$ and $V\left[G\right]$ in this interval of ordinals.

	For the $\con$ direction, first note that by the properties of Prikry
	forcing, $V\left[G\right]\vDash\mathrm{cf}(\eta)=\omega$ iff $V\vDash\mathrm{cf}(\eta)\in\left\{ \omega\right\} \cup\left\{ \kappa^{\alpha}\mid\alpha\in A\right\} $.
	We claim that this is iff
	\begin{align*}
		(*)\quad & M^{\chi}\left[\left\langle \hat{E}^{\gamma}\mid\gamma<\chi\right\rangle \right]\vDash                                                                                \\
		         & \quad\mathrm{cf}(\eta)\in\left\{ \omega\right\} \cup\bigcup\left\{ \hat{E}^{\gamma}\mid\gamma<\chi\right\} \cup\left\{ \sup\hat{E}^{\gamma}\mid\gamma<\chi\right\} .
	\end{align*}
	The direction $\impliedby$ is clear by the definitions of the sets
	$\hat{E}^{\gamma}$. As before, $V\vDash\mathrm{cf}(\eta)\in\left\{ \omega\right\} $
	implies $(*)$. By claim \ref{lem:k-cof-preserves} below, $V\vDash\mathrm{cf}(\eta)=\kappa^{\alpha}$
	implies $M^{\chi}\vDash\mathrm{cf}(\eta)=\kappa^{\alpha}$, which
	implies that
	\[
		M^{\chi}\left[\left\langle \hat{E}^{\gamma}\mid\gamma<\chi\right\rangle \right]\vDash\mathrm{cf}(\eta)=\kappa^{\alpha}
	\]
	as each sequence $\hat{E}^{\gamma}$ is a Prikry sequence on $\kappa_{\omega^{2}}^{\gamma}$
	over $M^{\chi}$, so it preserves the regularity of $\kappa^{\alpha}$.
	So $V\vDash\mathrm{cf}(\eta)\in\left\{ \omega\right\} \cup\left\{ \kappa^{\alpha}\mid\alpha\in A\right\} $
	indeed implies $(*)$, so all-in-all $V\left[G\right]\vDash\mathrm{cf}(\eta)=\omega$
	is equivalent to $\left(*\right)$, hence in $M^{\chi}\left[\left\langle \hat{E}^{\gamma}\mid\gamma<\chi\right\rangle \right]$
	we can calculate $\left(C^{*}\right)^{V\left[G\right]}$.\renewcommand{\qedsymbol}{$\square$\footnotesize{}{Theorem \ref{thm:C-star-in-Prikry-ext}}}
\end{proof}
\begin{claim}
	\label{lem:k-cof-preserves}For any $\alpha,\beta<\chi$, $\gamma\leq\omega^{2}$
	and $\eta\in\ord$, $M_{\gamma}^{\beta}\vDash\cof{\eta}=\kappa^{\alpha}$
	iff $V\vDash\cof{\eta}=\kappa^{\alpha}$.
\end{claim}

\begin{proof}
	\renewcommand{\qedsymbol}{$\square$\footnotesize{}{Claim \ref{lem:k-cof-preserves}}}$\kappa^{\alpha}$
	is regular in $V$, thus it is regular in every $M_{\gamma}^{\beta}$
	as it is an inner model of $V$. If $M_{\gamma}^{\beta}\vDash\cof{\eta}=\kappa^{\alpha}$,
	then there is a cofinal $\kappa^{\alpha}$-sequence in $\eta$ (in
	both $M_{\gamma}^{\beta}$ and $V$), and since $\kappa^{\alpha}$
	is regular we get $V\vDash\cof{\eta}=\kappa^{\alpha}$.

	Now assume $V\vDash\cof{\eta}=\kappa^{\alpha}$. The same argument
	rules out $M_{\gamma}^{\beta}\vDash\cof{\eta}<\kappa^{\alpha}$ so
	we need to rule out the case $M_{\gamma}^{\beta}\vDash\cof{\eta}>\kappa^{\alpha}$.
	If this were thee case, and $M_{\gamma}^{\beta}\vDash\lambda=\cof{\eta}$,
	then the sequence witnessing this is also in $V$, so the cofinality
	of $\lambda$ in $V$ must be $\kappa^{\alpha}$ as well.  So we
	focus on showing there is \textbf{no} $\lambda>\kappa^{\alpha}$ such
	that $V\vDash\cof{\lambda}=\kappa^{\alpha}$ while $M_{\gamma}^{\beta}\vDash\lambda$
	is regular. We prove this by induction on $\beta$ and $\gamma$.
	Consider a few cases.
	\begin{casenv}
		\item $\beta<\alpha$ or $\left(\beta,\gamma\right)=\left(\alpha,0\right)$.
		In both these cases $M_{\gamma}^{\beta}$ is an iteration of length
		$<\kappa^{\alpha}$ using measurables $<\kappa^{\alpha}$. As before,
		any $\lambda>\kappa^{\alpha}$ is represented by some $f:E=\kappa_{i_{1}}\times\dots\times\kappa_{i_{n}}\to V$
		for some iteration points $\kappa_{i_{1}},\dots,\kappa_{i_{n}}$,
		and by our construction (iterating measurables $<\kappa^{\alpha}$
		at most $\omega^{2}$ many times) they are all $<\kappa^{\alpha}$.
		Also note that $\kappa^{\alpha}=j_{\gamma}^{\beta}(\kappa^{\alpha})$
		by strong inaccessibility so $\lambda>\kappa^{\alpha}=j_{\gamma}^{\beta}(\kappa^{\alpha})$.
		If $M_{\gamma}^{\beta}\vDash\lambda$ is regular, we can assume that
		for every $\bar{\zeta}\in E$, $f\left(\bar{\zeta}\right)$ is a regular
		cardinal $>\kappa^{\alpha}$. If $V\vDash\cof{\lambda}=\kappa^{\alpha}$
		then we can find in $V$ a sequence $\left\langle \gamma_{\xi}\mid\xi<\kappa^{\alpha}\right\rangle $
		cofinal in $\lambda$ and a sequence of functions $\left\langle g_{\xi}\mid\xi<\kappa^{\alpha}\right\rangle $
		such that $g_{\xi}$ represents $\gamma_{\xi}$ in $M_{\gamma}^{\beta}$.
		For every $\xi<\kappa^{\alpha}$ we have $\left[g_{\xi}\right]<\left[f\right]$.
		For any $\bar{\zeta}\in E$, since $f\left(\bar{\zeta}\right)$ is
		regular $>\kappa^{\alpha}$, the set $\left\{ g_{\xi}(\bar{\zeta})\mid\xi<\kappa^{\alpha}\right\} \cap f(\bar{\zeta})$
		is bounded by some $g(\bar{\zeta})<f\left(\bar{\zeta}\right)$. The
		function $g$ represents in $M_{\gamma}^{\beta}$ an ordinal $<\lambda$,
		and we claim it bounds the sequence $\left\langle \gamma_{\xi}\mid\xi<\kappa^{\alpha}\right\rangle $.
		For every $\xi$, there is a large set (with respect to the ultrafilter
		$U^{E}$) where $g_{\xi}(\bar{\ensuremath{\zeta}})<f\left(\bar{\zeta}\right)$.
		But for any such $\bar{\zeta}$, $g_{\xi}\left(\bar{\zeta}\right)<g(\bar{\zeta})$,
		so indeed $\left[g_{\xi}\right]<\left[g\right]$. This contradicts
		the assumption that $\left\langle \gamma_{\xi}\mid\xi<\kappa^{\alpha}\right\rangle $
		is cofinal in $\lambda$.
		\item $\beta\geq\alpha$, $\gamma=\gamma'+1$, the statement is true for
		$M_{\gamma'}^{\beta}$, and we prove for $M_{\gamma}^{\beta}$. In
		this case $M_{\gamma}^{\beta}$ is the ultrapower of $M_{\gamma'}^{\beta}$
		with a measure on an ordinal $\kappa'\geq\kappa^{\alpha}$, so in
		particular $M_{\gamma}^{\beta}$ is closed under $\kappa^{\alpha}$
		sequences in $M_{\gamma'}^{\beta}$. So for any $\lambda>\kappa^{\alpha}$,
		if $V\vDash\cof{\lambda}=\kappa^{\alpha}$ then by induction hypothesis
		also $M_{\gamma'}^{\beta}\vDash\cof{\lambda}=\kappa^{\alpha}$ so
		also $M_{\gamma}^{\beta}\vDash\cof{\lambda}=\kappa^{\alpha}$, so
		in particular $\lambda$ is not regular in $M_{\gamma}^{\beta}$.
		\item $\left(\beta,\gamma\right)>_{\lex}\left(\alpha,0\right)$, $\gamma$
		is limit or $0$, and the statement is true for any $M_{\gamma'}^{\beta'}$
		such that $\beta'<\beta$ and $\gamma'\leq\omega^{2}$ or $\beta'=\beta$
		and $\gamma'<\gamma$. In this case $M_{\gamma}^{\beta}$ is the directed
		limit of the previous stages. Assume $V\vDash\cof{\lambda}=\kappa^{\alpha}$
		and let $\left\langle \alpha_{\xi}\mid\xi<\kappa^{\alpha}\right\rangle $
		be a cofinal sequence in $\lambda$. Let's re-index the iteration
		leading to $M_{\gamma}^{\beta}$ as $\left\langle N_{\zeta},j_{\zeta,\zeta'}\mid\zeta\leq\zeta'\leq\bar{\chi}\right\rangle $
		($\bar{\chi}=\beta+\gamma\leq\chi$) so $M_{\gamma}^{\beta}=N_{\bar{\chi}}$.
		By definition of the limit ultrapower, each $\alpha_{\xi}$ is of
		the form $j_{\zeta,\bar{\chi}}(\bar{\alpha}_{\xi})$ for some $\zeta<\bar{\chi}$
		and $\bar{\alpha}_{\xi}\leq\alpha_{\xi}$. We can also assume that
		each such $\zeta$ is large enough so that $\lambda\in\rng(j_{\zeta,\bar{\chi}})$.
		Since $\bar{\chi}<\kappa^{\alpha}$, there is some $\zeta$ fitting
		$\kappa^{\alpha}$ many $\alpha_{\xi}$s, so without loss of generality
		we can assume there is $\zeta$ that fits all of them. By our assumptions
		on $\gamma$ and $\beta$, we can also assume that $\zeta$ is large
		enough so that if $N_{\zeta}=M_{\gamma'}^{\beta'}$ then $\left(\beta',\gamma'\right)>_{\lex}\left(\alpha,0\right)$,
		so $\kappa^{\alpha}$ is a fixed point of $j_{\zeta,\bar{\chi}}$.
		If $\bar{\lambda}=\sup\left\{ \bar{\alpha}_{\eta}\mid\eta<\kappa\right\} $\footnote{Note that we can't assume this sequence is in $M_{\gamma}^{\beta}$}
		then, since $\lambda=\sup\left\{ j_{\zeta,\bar{\chi}}(\bar{\alpha}_{\xi})\mid\xi<\kappa^{\alpha}\right\} $
		and $\lambda\in\rng(j_{\zeta,\bar{\chi}})$, we must have that $\lambda=j_{\zeta,\bar{\chi}}(\bar{\lambda})$.
		$\bar{\lambda}$ is of cofinality $\kappa^{\alpha}$ in $V$, so by
		induction also in $N_{\zeta}$, hence by elementarity $N_{\bar{\chi}}=M_{\gamma}^{\beta}\vDash\cof{\lambda}=\kappa^{\alpha}$.\qedhere
	\end{casenv}
\end{proof}

\section{\label{sec:Inner-models-in-C-star}Inner models for short sequences
  of measures in $C^{*}$}

Our last task is to generalize theorem 5.6 of \cite{IMEL} from the
case of an inner model with one measurable to the case of an inner
model with a short sequence of measurables.
\begin{lem}
	\label{lem:sequences of measures-1} Assume $M=L\left[\mathcal{U}\right]$
	where $\mathcal{U}=\left\langle U^{\gamma}\mid\gamma<\chi\right\rangle $
	is in $M$ a short sequence of measures on the increasing sequence
	$\left\langle \kappa^{\gamma}\mid\gamma<\chi\right\rangle $. Denote
	for each $\gamma<\chi$ $\xi^{\gamma}=\big(\big(\kappa^{\gamma}\big)^{+}\big)^{M}$
	and let $f_{\mathcal{U}}:\chi\to2$ be defined by $f_{\mathcal{U}}(\gamma)=0\iff\mathrm{cf}^{V}(\xi^{\gamma})=\omega$.
	If $f_{\mathcal{U}}\in C^{*}$ then $C^{*}$ contains an iteration
	of $M$ by $\mathcal{U}$.
\end{lem}

\begin{proof}
	We combine the methods used in the previous sections -- from section
	\ref{sec:Short-mice-C-star} we take the method of recursively defining
	the sequence $\left\langle \lambda^{\gamma}\mid\gamma<\chi\right\rangle $
	which gives the lengths of the iterations, and from section \ref{sec:The-C-star-of}
	we take the method of recursively defining the sequence of iterated
	measures $\left\langle \bar{U}^{\gamma}\mid\gamma<\chi\right\rangle $.

	Let $\Theta$ be some regular cardinal (in $V$) greater than $\sup\left\{ \left(2^{\kappa^{\gamma}}\right)^{V}\mid\gamma<\chi\right\} $.
	We define simultaneously sequences $\left\langle \lambda^{\gamma}\mid\gamma<\chi\right\rangle $
	and $\left\langle \bar{U}^{\gamma}\mid\gamma<\chi\right\rangle $
	recursively such that $\lambda_{0}>\Theta$ and for every $\gamma<\chi$:
	\begin{enumerate}
		\item $\left\langle \lambda^{\gamma'}\mid\gamma'<\gamma\right\rangle $
		      and $\left\langle \bar{U}^{\gamma'}\mid\gamma'<\gamma\right\rangle $
		      are in $C^{*}$;
		\item $K^{\gamma}:=K\left[\left\langle \bar{U}^{\gamma'}\mid\gamma'<\gamma\right\rangle \right]=\left(K\left[\left\langle \bar{U}^{\gamma'}\mid\gamma'<\gamma\right\rangle \right]\right)^{C^{*}}$;
		\item $C^{*}\vDash\lambda^{\gamma}$ is a regular cardinal;
		\item $\gamma'<\gamma$ implies $\lambda^{\gamma'}<\lambda^{\gamma}$;
		\item $\forall\beta<\lambda^{\gamma}$ $\exists\eta\in\left(\beta,\lambda^{\gamma}\right)$
		      such that $K^{\gamma}\vDash\eta$ is a regular cardinal, and $\mathrm{cf}^{V}(\eta)=\omega\iff f_{\mathcal{U}}(\gamma)=1$
		      (so $\mathrm{cf}^{V}(\eta)=\omega\iff\mathrm{cf}^{V}(\xi^{\gamma})\ne\omega$).
	\end{enumerate}
	Assume we've defined $\left\langle \lambda^{\gamma'}\mid\gamma'<\gamma\right\rangle ,\left\langle \bar{U}^{\gamma'}\mid\gamma'<\gamma\right\rangle \in C^{*}$
	for some $\gamma<\chi$. By essentially the same argument as in the
	proof of theorem \ref{thm:mice in C*} we show that every regular
	cardinal in $V$ satisfies 3-5, where the main differences are that
	here in $K^{\gamma}$ we take the iterations of the measures -- $\left\langle \bar{U}^{\gamma'}\mid\gamma'<\gamma\right\rangle $
	-- instead of the original ones, and that we need to use $f_{\mathcal{U}}$
	to know which iteration points to choose, instead of the single $\delta_{M}$
	we had there. So in $C^{*}$ we can define $\lambda_{\gamma}$ to
	be the first ordinal $>\sup\left\langle \lambda^{\gamma'}\mid\gamma'<\gamma\right\rangle $
	(or $>\Theta$ if $\gamma=0$) satisfying 3-5. Again by the same
	argument as before we get that in fact $\lambda_{\gamma}$ is the
	image of $\kappa^{\gamma}$ under the iteration of $M$ by the measures
	$\left\langle U^{\gamma'}\mid\gamma'\leq\gamma\right\rangle $ according
	to $\left\langle \lambda^{\gamma'}\mid\gamma'<\gamma\right\rangle ^{\frown}\left\langle \theta_{\gamma}\right\rangle $,
	for some $\theta_{\gamma}$ i.e. $\lambda^{\gamma}=\kappa_{\theta_{\gamma}}^{\gamma}$.
	Now we let
	\[
		E^{\gamma}=\begin{cases}
			\left\{ \eta<\lambda^{\gamma}\mid\eta\text{ is regular in \ensuremath{K^{\gamma}} and }\mathrm{cf}^{V}\left(\eta\right)=\omega\right\} & f_{\mathcal{U}}=1 \\
			\left\{ \eta<\lambda^{\gamma}\mid\eta\text{ is regular in \ensuremath{K^{\gamma}} }\mathrm{cf}^{V}\left(\eta\right)\ne\omega\right\}   & f_{\mathcal{U}}=0
		\end{cases}
	\]
	and get (using the appropriate variations of claims \ref{claim:characterize-regulars}
	and \ref{claim:regulars-agree}) that $E^{\gamma}\in C^{*}$ generates
	the iterated measure $\bar{U}^{\gamma}$ on $\lambda^{\gamma}=\kappa_{\theta_{\gamma}}^{\gamma}$.

	Since all the definitions were done in a recursive way inside $C^{*}$,
	we can proceed through limit stages up to $\chi$ to get that the
	sequences are in $C^{*}$, hence the iteration of $M$ by $\left\langle \theta_{\gamma}\mid\gamma<\chi\right\rangle $
	(which is not in necessarily in $C^{*}$) is in fact determined by
	$\left\langle \lambda^{\gamma}\mid\gamma<\chi\right\rangle $ and
	$\left\langle \bar{U}^{\gamma}\mid\gamma<\chi\right\rangle $ which
	are in $C^{*}$, so this iteration is in $C^{*}$.
\end{proof}
\begin{cor}
	If there is an inner model with a short sequence of measures of order
	type $\chi$ then there is such an inner model inside $C^{*}$.
\end{cor}

\begin{proof}
	Let $M=L\left[\mathcal{U}\right]$ where $\mathcal{U}$ is a short
	sequence of measures of order type $\chi$, and let $f_{\mathcal{U}}:\chi\to2$
	be defined by $f_{\mathcal{U}}(\gamma)=0\iff\mathrm{cf}^{V}(\xi^{\gamma})=\omega$.
	We will show that there is $\tilde{\mathcal{U}}$ which is also of
	order type $\chi$ such that $f_{\tilde{\mathcal{U}}}\in C^{*}$,
	so we'll be done by the previous lemma applied to $L[\mathcal{\tilde{U}}]$.
	\begin{lem}
		Let $\alpha=\omega^{\beta}$, and let $f:\alpha\to2$. Then there
		is $A\con\alpha$, $\otp A=\alpha$ such that $\left|f``A\right|=1$.
	\end{lem}

	\begin{proof}
		Otherwise $\alpha$ is the disjoint union of two sets of order-type
		$<\alpha$, $A_{i}=f^{-1}\left\{ i\right\} $ for $i=0,1$. We prove
		that this is impossible. If $\beta=0$ then $\alpha=1$ so this is
		clear. If $\beta=\gamma+1$, consider $\omega^{\beta}=\omega^{\gamma}\cdot\omega$.
		So $\max\left\{ \otp A_{0},\otp A_{1}\right\} <\omega^{\gamma}\cdot n$
		for some $n$, so $\otp(A_{0}\cup A_{1})<\omega^{\gamma}\cdot n+\omega^{\gamma}\cdot n=\omega^{\gamma}\cdot2n<\omega^{\gamma}\cdot\omega=\omega^{\beta}$
		by contradiction. For $\beta$ limit, $\omega^{\beta}=\sup\left\{ \omega^{\gamma}\mid\gamma<\beta\right\} $,
		$\max\left\{ \otp A_{0},\otp A_{1}\right\} <\omega^{\gamma}$ for
		some $\gamma$. So $\otp(A_{0}\cup A_{1})<\omega^{\gamma}+\omega^{\gamma}<\omega^{\gamma+1}<\omega^{\beta}$
		by contradiction.
	\end{proof}
	Now by Cantor normal form, there are $\beta_{0}\geq\beta_{1}\geq\dots\geq\beta_{n}\geq0$
	such that $\chi=\omega^{\beta_{0}}+\dots+\omega^{\beta_{n}}$, so
	$\chi$ is partitioned into the finitely many successive segments
	\begin{align*}
		I_{0} & =\left[0,\omega^{\beta_{0}}\right).                                                                     \\
		I_{1} & =\left[\omega^{\beta_{0}},\omega^{\beta_{0}}+\omega^{\beta_{1}}\right)                                  \\
		      & \vdots                                                                                                  \\
		I_{n} & =\left[\omega^{\beta_{0}}+\dots+\omega^{\beta_{n-1}},\omega^{\beta_{0}}+\dots+\omega^{\beta_{n}}\right)
	\end{align*}
	Each segment $I_{k}$ is order isomorphic to the corresponding $\omega^{\beta_{k}}$,
	so there is $A_{k}\con I_{k}$ such that $\otp A_{k}=\omega^{\beta_{k}}$
	and $\left|f_{\mathcal{U}}``A_{k}\right|=1$. Clearly $\otp\left(\bigcup_{i=0}^{n}A_{k}\right)=\chi$
	so let $\phi:\chi\to\bigcup_{i=0}^{n}A_{k}$ be the order isomorphism
	and define $\tilde{\mathcal{U}}$ by $\tilde{U}_{\gamma}=U_{\phi(\gamma)}$.
	Then $f_{\mathcal{\tilde{U}}}=\left(f_{\mathcal{U}}\mets\bigcup_{i=0}^{n}A_{k}\right)\circ\phi$
	is simply the finite union of constant functions, so it is indeed
	in $C^{*}$ as required.
\end{proof}

\section{Further generalizations and open questions}

In this paper we dealt only with short sequences of measure, and in
fact used shortness a few times. A natural question is whether similar
results can be obtained also for longer sequences of measures $\mathcal{U}=\left\langle U^{\gamma}\mid\gamma<\chi\right\rangle $
where we do not require $\chi<\kappa^{0}$, and whether we can allow
also measurable limits of measurables, or even measurables of higher
Mitchel order.

In the forthcoming \cite{welch2022C*} Welch shows that as long as
there are no measurable limits of measurables, the same characterization
holds. To be precise, denote by $O^{k}$ (O-kukri) the assumption
that there is an elementary embedding of an inner model with a proper
class of measurables to itself (cf. \cite{welch2022closed}). Then
Welch shows:
\begin{thm*}
	Assume $\neg O^{k}$. If $V=L[\vec{\mu}]=L[\left\langle \mu_{\alpha}\mid\alpha<\tau\right\rangle ]$
	where $\tau\leq\ord$, and $\tilde{M}$ is the model obtained by iterating
	every $\mu_{\alpha}$ $\omega^{2}$ many times, then $C^{*}=\tilde{M}[\left\langle c(\alpha)\mid\alpha<\tau\right\rangle $where
	$c(\alpha)$ is the $\omega$-subsequence of limit points of the $\omega^{2}$-sequence
	of critical points leading to the $\alpha$th measurable of $\tilde{M}$.
\end{thm*}
In \cite{welch2022closed} Welch also proves that if $O^{k}$ exists,
then it is in $C^{*}$, so the question of generalizing this result
to the case of measurable limit of measurables is still open. It is
also still unknown whether there can, relative to any large cardinal
hypothesis, be a measurable \emph{in} $C^{*}$ itself.

\bibliographystyle{amsplain}
\bibliography{Bibliography}

\providecommand{\bysame}{\leavevmode\hbox to3em{\hrulefill}\thinspace}
\providecommand{\MR}{\relax\ifhmode\unskip\space\fi MR }
\providecommand{\MRhref}[2]{%
  \href{http://www.ams.org/mathscinet-getitem?mr=#1}{#2}
}
\providecommand{\href}[2]{#2}
\begin{thebibliography}{1}

\bibitem{IMEL}
Juliette Kennedy, Menachem Magidor, and Jouko V{\"a}{\"a}n{\"a}nen, \emph{Inner
  models from extended logics: Part 1}, Journal of Mathematical Logic
  \textbf{21} (2021), no.~02, 2150012.

\bibitem{koepkedoc}
Peter Koepke, \emph{A theory of short core models and applications}, Ph.D.
  thesis, University of Freiburg, 1983.

\bibitem{koepke1988}
\bysame, \emph{Some applications of short core models}, Annals of pure and
  applied logic \textbf{37} (1988), no.~2, 179--204.

\bibitem{lucke-schlicht}
Philipp L{\"u}cke and Philipp Schlicht, \emph{Measurable cardinals and good
  ${\Sigma_1(\kappa)}$-wellorderings}, Mathematical Logic Quarterly \textbf{64}
  (2018), no.~3, 207--217.

\bibitem{mitchell1984core}
William~J. Mitchell, \emph{The core model for sequences of measures. {I}},
  Mathematical Proceedings of the Cambridge Philosophical Society, vol.~95,
  Cambridge University Press, 1984, pp.~229--260.

\bibitem{handbook-finestructure}
Ralf Schindler and Martin Zeman, \emph{Fine structure}, Handbook of set theory,
  Springer, 2010, pp.~605--656.

\bibitem{welch2022C*}
Philip~D. Welch, \emph{{$C^*$ in $L[E]$-models below $O^k$}}, Personal
  communication (2022).

\bibitem{welch2022closed}
\bysame, \emph{Closed and unbounded classes and the {H\"artig} quantifier
  model}, The Journal of Symbolic Logic \textbf{87} (2022), no.~2, 564--584.

\bibitem{yaar-iterating-Cstar}
Ur~Ya'ar, \emph{Iterating the cofinality-{$\omega$} constructible model}, The
  Journal of Symbolic Logic \textbf{88}, no.~4, 1682--1691.

\end{thebibliography}

\end{document}